\documentclass[12pt]{amsart}
\usepackage{amsmath}
\usepackage{amsfonts}
\usepackage{amssymb}
\usepackage{geometry}
\usepackage{amsthm}
\usepackage{url}
\usepackage{hyperref}
\usepackage{cite}
\usepackage{amsrefs}
\usepackage{indentfirst}
\usepackage{epsfig}
\usepackage{graphicx}
\numberwithin{equation}{section}
\usepackage{cite}
\input xy
\xyoption{all}

\theoremstyle{plain}
\newtheorem{prop}{Proposition}[section]

\newtheorem{theo}[prop]{Theorem}
\newtheorem{coro}[prop]{Corollary}

\newtheorem{lemm}[prop]{Lemma}
\newtheorem{assu}[prop]{Assumption}

\theoremstyle{definition}
\newtheorem{defi}[prop]{Definition}

\newtheorem{conj}[prop]{Conjecture}

\newtheorem{rema}[prop]{Remark}

\newtheorem{nota}[prop]{Notation}

\newcommand{\bZ}{\mathbb Z}

\newcommand{\cD}{\mathcal D}

\newcommand{\Spec}{{\rm Spec}}

\newcommand{\Pic}{{\rm Pic}}
\newcommand{\Val}{{\rm Val}}
\newcommand{\Gal}{{\rm Gal}}

\newcommand{\Vol}{{\rm Vol}}

\begin{document}
\title[]
{Campana points on biequivariant compactifications of the Heisenberg group}

\author[]{Huan Xiao}
\address{Department of Mathematics, Faculty of Science, Kumamoto University, Kumamoto, Japan}
\email{197d9003@st.kumamoto-u.ac.jp}
\date{\today}
\subjclass[2010]{11G50;  11G35; 14G05; 14G10}
\keywords{Campana points; the Heisenberg group}

\begin{abstract} 

We study Campana points on biequivariant compactifications of the Heisenberg group and confirm the log Manin conjecture introduced by Pieropan,  Smeets, Tanimoto and V\'{a}rilly-Alvarado.

\end{abstract}

\maketitle

\setcounter{tocdepth}{1}
\tableofcontents

\section{Introduction}
Manin’s conjecture concerns the distribution of rational points of bounded height on Fano varieties over number fields and has been extensively studied. It was initially proposed by Franke, Manin and Tschinkel \cite{fmt}. Later on the current and more appropriate formulations of Manin’s conjecture were made in \cite{batyrev-manin, bt-tamagawa, LST18, peyre95}. Let $ X $ be a smooth Fano variety defined over a number field $ F $ and $ L $ an ample line bundle on $ X $. Let $ \mathsf{H}_{\mathcal{L}} $ be a height function
\begin{equation*}
\mathsf{H}_{\mathcal{L}}: X(F)\rightarrow \mathbb{R}_{> 0},
\end{equation*}
where $ \mathcal{L} $ is an adelically metrized line bundle associated to $ L $. Manin’s conjecture states that there is a subset $ U $ of $ X(F) $ such that the counting function
\begin{equation*}
\mathsf{N}(U,\mathcal{L},T):=\#\{x\in U\mid \mathsf{H}_{\mathcal{L}}(x)\leq T\}
\end{equation*}
satisfies the asymptotic formula
\begin{equation*}
\mathsf{N}(U,\mathcal{L},T)\sim cT^{a(X,L)}(\log T)^{b(X,F,L)-1},
\end{equation*}
as $ T\rightarrow \infty $, where $ c $  is a positive constant, and $ a(X,L), b(X,F,L) $ are certain geometric invariants introduced in \cite{batyrev-manin}. Manin's conjecture for homogeneous spaces has been studied for example in \cite{BT98, BT-general, VecIII, STBT, GMO08, GO11, TT12, GTBT11, ST15}, and it also has been extensively studied for various other classes of varieties, such as Del Pezzo surfaces. We refer the readers to the excellent surveys \cite{t-survey, browning-survey} and the references therein. Besides the distribution of rational points, the problem of counting integral points has also been considered by \cite{CT-additive, CT-toric, BO12, TBT11, TT15, Cho19}.\par
Let $ G $ be the Heisenberg group, that is, the subgroup of the general linear group $ \mathrm{GL}_{3} $ whose elements are the upper triangular matrices of the form
\begin{equation*}
\begin{pmatrix}
1 & x &z\\ 
0& 1 & y\\
0& 0 & 1
\end{pmatrix}.
\end{equation*}
As an algebraic variety $ G $ is isomorphic to a 3-dimensional affine space. A smooth projective variety $ X $ is called a biequivariant compactification of $ G $ if $ G $ is a dense Zariski open subset of $ X $ and $ X $ carries both left and right $ G $-actions, extending the left and right actions of $ G $ on itself. A trivial example of biequivariant compactifications of $ G $ is the projective space $ \mathbb{P}^{3} $: let $ [a:b:c:d]\in \mathbb{P}^{3} $ be a point. Then the left action of $ G $ is
\begin{equation*}
\begin{pmatrix}
1 & x &z\\ 
0& 1 & y\\
0& 0 & 1
\end{pmatrix}\cdot [a:b:c:d]=[a:ax+b:ay+c:az+d+xc],
\end{equation*}
and the right action is 
\begin{equation*}
[a:b:c:d]\cdot\begin{pmatrix}
1 & x &z\\ 
0& 1 & y\\
0& 0 & 1
\end{pmatrix}=[a:ax+b:ay+c:az+by+d].
\end{equation*}
For a general construction one may take $ X $ to be the Zariski closure of an orbit of a projective representation of $ G $, for details see \cite{ST04}.\par
Campana orbifolds and Campana points were introduced by Campana \cite{Cam, Cam11}, and the study of the distribution of Campana points over a number field was initiated in for example \cite{BVV12, VV12, BY19}. Recently, M. Pieropan, A. Smeets, S. Tanimoto and A. V\'{a}rilly-Alvarado \cite{PSTVA19} initiated a systematic quantitative study of Campana points. In their paper, Pieropan et al.~\cite{PSTVA19} studied the distribution of Campana points of bounded height on equivariant compactifications of vector groups and proposed a log version of Manin's conjecture. Pieropan and Schindler \cite{PS20} recently developed the hyperbola method to study the distribution of Campana points on toric varieties over $ \mathbb{Q} $.\par
In this paper we study Campana points on biequivariant compactifications of the Heisenberg group and confirm the log Manin conjecture of Pieropan et al.~\cite{PSTVA19} for this class of varieties over $ \mathbb{Q} $.\par
Let us briefly review the notion of Campana points and the log Manin conjecture studied in Pieropan et al.~ \cite{PSTVA19}. We begin by recalling some basics of Campana orbifolds and Campana points.
\subsection{Campana orbifolds}
The notion of Campana orbifolds was introduced in Campana \cite{Cam, Cam11} and Abramovich \cite{Abra}. Pieropan et al.~\cite[\S 3.2]{PSTVA19} clarified two types of Campana points, and in this paper we shall discuss the Campana points in the sense of \cite[Definition 3.4]{PSTVA19}. Through the paper we let $ F $ be a number field.\par
Let $ X $ be a smooth variety and $ D_{\varepsilon} $ an effective Weil $ \mathbb{Q} $-divisor  on $  X$. A pair $(X, D_{\varepsilon})$ is called a $\textit{Campana orbifold}$ over $ F $ if
\begin{equation*}
D_{\varepsilon}=\sum_{\alpha\in \mathcal{A}}\varepsilon_{\alpha}D_{\alpha},
\end{equation*}
where $ D_{\alpha} $ is a prime divisor and
\begin{equation*}
\varepsilon_{\alpha}\in \mathfrak{W}:=\left\lbrace 1-\dfrac{1}{m}\vline \: m\in \mathbb{Z}_{\geq 1}\right\rbrace \cup \{1\}
\end{equation*}
for all $ \alpha\in \mathcal{A} $, and the support $ D_{\mathrm{red}}=\sum_{\alpha\in \mathcal{A}}D_{\alpha} $ is a divisor with strict normal crossings.\par
One sees that any Campana orbifold $(X, D_{\varepsilon})$ is a $\mathsf{dlt} $  (divisorial log terminal) pair. If moreover $ \varepsilon_{\alpha}\neq 1 $ for all $ \alpha\in \mathcal{A} $, we say that $(X, D_{\varepsilon})$ is $\mathsf{klt} $ (Kawamata log terminal). For the definitions of dlt and klt one should consult \cite{KM98}.\par 
Let $ \Val(F) $ be the set of all places of $ F $. For $ v\in\Val(F) $, we denote the completion of $ F $ at $ v $ by $ F_{v} $. If $ v $ is a nonarchimedean place, we denote the corresponding ring of integers by $ \mathcal{O}_{v} $ with  maximal ideal $ \mathfrak{m}_{v} $ and residue field $ k_{v} $. We denote the ring of adeles of $ F $ by $ \mathbb{A}_{F} $ and fix a finite set $ S $ of places of $  F$ containing all archimedean places.
\begin{defi}[{\cite[\S 3]{PSTVA19}}]
We say $ (\mathcal{X}, \mathcal{D}_{\varepsilon}) $ is a $ \textit{good integral model} $ of $(X, D_{\varepsilon})$ away from $S $ if $ \mathcal{X} $ is a flat, regular and proper model over $ \Spec\, \mathcal{O}_{F,S} $ where $ \mathcal{D}_{\varepsilon}=\sum_{\alpha\in \mathcal{A}}\varepsilon_{\alpha}\mathcal{D}_{\alpha} $ and $ \mathcal{D}_{\alpha} $ is the Zariski closure of $ D_{\alpha} $ in $ \mathcal{X} $.
\end{defi}

\subsection{Campana points}
Let us fix a good integral model $(\mathcal{X}, \mathcal{D}_{\varepsilon})$ for $(X, D_{\varepsilon})$ over $ \Spec\,\mathcal{O}_{F,S} $, and let $ \mathcal{A}_{\varepsilon}=\{\alpha\in \mathcal{A}: \varepsilon_{\alpha}\neq 0\} $ and $ X^{\circ}=X\backslash (\cup_{\alpha\in \mathcal{A}_{\varepsilon}}D_{\alpha}) $. If $ P\in X^{\circ}(F) $ and $ v\not\in S $ then there is an induced point $ \mathcal{P}_{v}\in \mathcal{X}(\mathcal{O}_{v}) $. Let $ \alpha\in \mathcal{A} $ be such that $ \mathcal{P}_{v}\nsubseteq \mathcal{D}_{\alpha}$, we denote the colength of the ideal defined by the pullback of $ \mathcal{D}_{\alpha} $ via $ \mathcal{P}_{v} $ by $ n_{v}(\mathcal{D}_{\alpha},P) $ and call it the intersection multiplicity of $ P $ and $ \mathcal{D}_{\alpha} $ at $ v $.\par
Following \cite{PSTVA19}, a point $ P\in X^{\circ}(F) $ is called a $\textit{Campana $ \mathcal{O}_{F,S} $-point} $ on $(\mathcal{X}, \mathcal{D}_{\varepsilon})$ if for all places $ v\not\in S $,  $ n_{v}(\mathcal{D}_{\alpha},P)=0 $ whenever $ \alpha\in \mathcal{A}_{\varepsilon} $ satisfies $ \varepsilon_{\alpha}=1 $, and 
\begin{equation*}
n_{v}(\mathcal{D}_{\alpha},P)\geq \dfrac{1}{1-\varepsilon_{\alpha}}
\end{equation*}
whenever $ \alpha\in \mathcal{A}_{\varepsilon} $ satisfies $ \varepsilon_{\alpha}<1 $ and  $ n_{v}(\mathcal{D}_{\alpha},P)>0 $.
So when writing $ \varepsilon_{\alpha}=1-\frac{1}{m_{\alpha}} $, $ n_{v}(\mathcal{D}_{\alpha},P)\geq m_{\alpha} $ whenever $ n_{v}(\mathcal{D}_{\alpha},P)>0 $.
We denote by $(\mathcal{X}, \mathcal{D}_{\varepsilon})(\mathcal{O}_{F,S})$ the set of all Campana $ \mathcal{O}_{F,S} $-points on $(\mathcal{X}, \mathcal{D}_{\varepsilon})$.

\subsection{A log Manin conjecture}
Let $ (X,D_{\varepsilon}) $ be a $ \textit{Fano orbifold} $ over a number
field $ F $, that is, $ (X,D_{\varepsilon}) $ is a Campana orbifold with $ X $ projective and $ -(K_{X}+D_{\varepsilon}) $ ample, where $ K_{X} $ is the canonical divisor of $ X $. Assume that $ (X,D_{\varepsilon}) $ is klt. Let
\begin{equation*}
\mathsf{H}_{\mathcal{L}}: X(F)\rightarrow \mathbb{R}_{>0}
\end{equation*}
be the height function associated with an adelically metrized big divisor $ \mathcal{L}=(L,\Vert \cdot \Vert )$ on $ X $. Let $U \subset X(F)$ and $ T>0 $, we denote the counting function by
\begin{equation*}
\mathsf{N}(U,\mathcal{L},T)=\#\{P\in U\vert \mathsf{H}_{\mathcal{L}}(P)\leq T\}.
\end{equation*}
\begin{conj}[Log Manin conjecture]\label{conj}
Suppose that the divisor $ L $ is big and nef and that the set of klt Campana points $ (\mathcal{X}, \mathcal{D}_{\varepsilon})(\mathcal{O}_{F,S}) $ is not thin in the sense of \cite[Definition 3.6]{PSTVA19}. Then there is a thin set $ Z\subset (\mathcal{X}, \mathcal{D}_{\varepsilon})(\mathcal{O}_{F,S}) $ such that as $ T\rightarrow \infty $
\begin{equation*}
\mathsf{N}((\mathcal{X}, \mathcal{D}_{\varepsilon})(\mathcal{O}_{F,S})\backslash Z,\mathcal{L},T)\sim c(F,S,(\mathcal{X}, \mathcal{D}_{\varepsilon}),\mathcal{L},Z)T^{a((X,D_{\varepsilon}),L)}\left(\log T\right)^{b(F,(X,D_{\varepsilon}),L)-1}
\end{equation*}
where $ a((X,D_{\varepsilon}),L),b(F,(X,D_{\varepsilon}),L) $ and $ c(F,S,(\mathcal{X}, \mathcal{D}_{\varepsilon}),\mathcal{L},Z) $ are positive constants described in \cite{PSTVA19}.
\end{conj}
Pieropan et al.~ \cite[Theorem 1.2]{PSTVA19} proved the conjecture above for equivariant compactifications of vector groups, and they also discussed dlt cases \cite[Theorem 1.4]{PSTVA19}.

\subsection{Main results in this paper}
It is shown in \cite{ST04} that Manin's conjecture is true for biequivariant compactifications of the Heisenberg group. Following the strategy of Pieropan et al.~ \cite{PSTVA19}, we study Campana points on biequivariant compactifications of the Heisenberg group over $ \mathbb{Q} $ and we confirm the above log Manin conjecture for this case. 
\begin{theo}\label{main thm}
Let $ X $ be a smooth projective biequivariant compactification of the Heisenberg group $ G $ over $ \mathbb{Q} $ such that the boundary divisor $D = X \backslash G$ is a strict normal crossings divisor on $ X $. Let $ S $ be a finite set of places of $ \mathbb{Q} $ containing the archimedean place. Assume that $ (X,D_{\varepsilon}) $ is klt and has a good integral model away from $ S $ and assume Assumption \ref{assu}(see in \S \ref{ass}). If $ aL+K_{X}+D_{\varepsilon} $ is rigid (i.e., its Iitaka dimension is 0), then Conjecture \ref{conj} holds for $ (\mathcal{X}, \mathcal{D}_{\varepsilon},\mathcal{L}) $ with the exceptional set $ Z=(X\backslash G)\cap (\mathcal{X}, \mathcal{D}_{\varepsilon})(\mathbb{Z}_{S})$.
\end{theo}
Let us introduce some notations. We denote
\begin{equation*}
G(\mathbb{Q})_{\varepsilon}=G(\mathbb{Q})\cap (\mathcal{X}, \cD_{\varepsilon})(\mathbb{Z}_S),
\end{equation*}
and
\begin{equation*}
G(\mathbb{Q}_{v})_{\varepsilon}=G(\mathbb{Q}_{v})\cap (\mathcal{X}, \cD_{\varepsilon})(\mathbb{Z}_{v}).
\end{equation*}
For $ v\not \in S $, let $ \delta_{\varepsilon,v} $ denote the local indicator function detecting whether or not a given point in $ G(\mathbb{Q}_{v}) $ belongs to $ G(\mathbb{Q}_{v})_{\varepsilon} $. If $ v\in S $, set $ \delta_{\varepsilon,v}\equiv 1 $. The global indicator function is thus $ \delta_{\varepsilon}=\prod_{v} \delta_{\varepsilon,v} $.

\begin{theo}\label{main thm2}
Let $ \mathcal{X},\mathcal{D},\varepsilon, S $ be as in the theorem above and let $ L=-\left(K_{X}+D_{\varepsilon}\right) $. Assume that $ (X,D_{\varepsilon}) $ is dlt and has a good integral model away from $ S $ and assume Assumption \ref{assu}(see in \S \ref{ass}). Set
\begin{equation*}
G(\mathbb{Q})_{\varepsilon}=G(\mathbb{Q})\cap (\mathcal{X}, \cD_{\varepsilon})(\mathbb{Z}_S).
\end{equation*}
Then as $ T\rightarrow \infty $, there are positive constants $ b,c $ such that
\begin{equation*}
\mathsf{N}(G(\mathbb{Q})_{\varepsilon},\mathcal{L},T)\sim \dfrac{c}{(b-1)!}T(\log T)^{b-1}.
\end{equation*}
\end{theo}
\begin{rema}
One can achieve Assumption \ref{assu} by applying a resolution, however one cannot reduce to this situation as the set of Campana points is not invariant under birational morphisms. In particular, the coefficient  $ c(F,S,(\mathcal{X}, \mathcal{D}_{\varepsilon}),\mathcal{L},Z) $ in Conjecture \ref{conj} will change in general under birational morphisms. For details and examples see \cite[\S 3.6]{PSTVA19}.
\end{rema}
To prove Theorems \ref{main thm} and \ref{main thm2}, we use the height zeta function method. As we shall see later, the height zeta function is defined to be
\begin{equation*}
\mathsf{Z}_{\varepsilon}(\mathbf{s},g)=\sum_{\gamma\in G(\mathbb{Q})}\delta_{\varepsilon}(\gamma g)\mathsf{H}(\mathbf{s},\gamma g)^{-1},
\end{equation*}
where $ g\in G(\mathbb{A}_{\mathbb{Q}}) $, for the meaning of $ \mathbf{s} $ see Notation \ref{notations} and for the meaning of the height pairing see \S \ref{height pairing}. We will not use this general height zeta function. Instead for our convenience we use the restriction of $ \mathsf{Z}_{\varepsilon}(\mathbf{s},g) $ to the identity $ g=e\in G(\mathbb{A}_{\mathbb{Q}}) $:
\begin{equation*}
\mathsf{Z}_{\varepsilon}(\mathbf{s})=\sum_{\gamma\in G(\mathbb{Q})}\delta_{\varepsilon}(\gamma)\mathsf{H}(\mathbf{s},\gamma )^{-1}.
\end{equation*}
We consider the representation theory of the Heisenberg group in the adelic setting \cite{ST04} and the spectral decomposition of a certain representation space gives:
\begin{equation*}
\mathsf{Z}_{\varepsilon}(\mathbf{s})=\mathsf{Z}_{0,\varepsilon}(\mathbf{s})+\mathsf{Z}_{1,\varepsilon}(\mathbf{s})+\mathsf{Z}_{2,\varepsilon}(\mathbf{s}),
\end{equation*}
where 
\begin{equation*}
\mathsf{Z}_{0,\varepsilon}(\mathbf{s})=\int_{G(\mathbb{A}_{\mathbb{Q}})}\mathsf{H}(\mathbf{s},g)^{-1}\delta_{\varepsilon}(g)\mathrm{d}g,
\end{equation*}
\begin{equation*}
\mathsf{Z}_{1,\varepsilon}(\mathbf{s})=\sum_{\eta}\int_{G(\mathbb{A}_{\mathbb{Q}})}\mathsf{H}(\mathbf{s},g)^{-1}\overline{\eta}(g)\delta_{\varepsilon}(g)\mathrm{d}g,
\end{equation*}
\begin{equation*}
\mathsf{Z}_{2,\varepsilon}(\mathbf{s})=\sum_{\psi}\sum_{\omega^{\psi}}\omega^{\psi}(e)\int_{G(\mathbb{A}_{\mathbb{Q}})}\mathsf{H}(\mathbf{s},g)^{-1}\overline{\omega}^{\psi}(g)\delta_{\varepsilon}(g)\mathrm{d}g.
\end{equation*}
For details of the above spectral decomposition see Lemma \ref{lemma decomposition}. We obtain a meromorphic continuation of the function $ \mathsf{Z}_{\varepsilon}(\mathbf{s}) $ and then apply Tauberian theorem to derive our results. The treatment of $ \mathsf{Z}_{0,\varepsilon}(\mathbf{s}) $ and $ \mathsf{Z}_{1,\varepsilon}(\mathbf{s}) $ is essentially analogous to that of \cite{PSTVA19}. For $ \mathsf{Z}_{2,\varepsilon}(\mathbf{s}) $, we will use the theta distribution and Schwartz-Bruhat functions described in \cite[\S 3.7]{ST04} as a tool to compute and estimate it.\par 
In this paper for simplicity we only consider the Heisenberg group over $ \mathbb{Q} $. We hope in the future to treat general unipotent groups over an arbitrary number field using the orbit method developed by Shalika-Tschinkel \cite{ST15}.
\subsection{Acknowledgement}
I would like to thank my advisor Prof. Sho Tanimoto for valuable discussion on the paper. I am grateful to Prof. Marta Pieropan and Prof. Anthony V\'{a}rilly-Alvarado for their comments on an earlier draft. I thank the referee for detailed and helpful comments. The author is partially supported by China Scholarship Council.

\section{Geometry of biequivariant compactifications of the Heisenberg group}
In this section we recall some basic facts on the geometry of biequivariant compactifications of the Heisenberg group from Shalika and Tschinkel \cite{ST04}. Hereafter we suppose $ F=\mathbb{Q} $. We denote the points of the Heisenberg group $ G $ by
\begin{equation*}
g=g(x,z,y)=\begin{pmatrix}
1 & x &z\\ 
0& 1 & y\\
0& 0 & 1
\end{pmatrix}
\end{equation*}
where $ x,y,z $ are coordinates for a 3-dimensional affine space. Let $ X $ be a smooth projective  biequivariant compactification of $ G $ with boundary $ D=X\backslash G $ consisting of irreducible components $ D=\cup_{\alpha\in \mathcal{A}}D_{\alpha} $, with strict normal crossings. We write the anticanonical divisor of $ X $ as $ -K_{X}=\sum_{\alpha\in \mathcal{A}}\kappa_{\alpha}D_{\alpha} $, denote the cone of effective divisors on $ X $ by $ \Lambda_{\mathrm{eff}}(X) $ and denote by $ \Pic(X) $ the Picard group of $ X $.
\begin{prop}[{\cite[Proposition 1.5]{ST04}}]
With the notations above, we have
\begin{enumerate}
\item \begin{equation*}
\Pic(X)=\oplus_{\alpha\in \mathcal{A}}\mathbb{Z}D_{\alpha},
\end{equation*}
\item \begin{equation*}
\Lambda_{\mathrm{eff}}(X)=\oplus_{\alpha\in \mathcal{A}}\mathbb{R}_{\geq 0}D_{\alpha},
\end{equation*}
\item $ \kappa_{\alpha}\geq 2 $ for all $ \alpha\in \mathcal{A} $.
\end{enumerate}
\end{prop}
\begin{coro}[{\cite[Corollary 1.7]{ST04}}]\label{coro divisor}
Let $ F[G] $ be the coordinate ring of $ G $ over $ F $. The divisor of every non-constant function $ f\in F[G] $ can be written as
\begin{equation*}
\mathrm{div}(f) = E(f)-\sum_{\alpha\in \mathcal{A}}d_{\alpha}(f)D_{\alpha}
\end{equation*}
where $ E(f) $ is the divisor of $ \{f=0\} $ in $ G $ and $ d_{\alpha}(f)\geq 0 $ for all $ \alpha\in \mathcal{A} $. Moreover, there is at least one $ \alpha\in \mathcal{A} $ such that $ d_{\alpha}(f)>0 $.
\end{coro}
\begin{nota}\label{notations}
Let $ \Pic(X)_{\mathbb{C}}:=\Pic(X)\otimes_{\mathbb{Z}}\mathbb{C} $. We introduce coordinates on $ \Pic(X)_{\mathbb{C}} $ using the basis $ (D_{\alpha})_{\alpha\in \mathcal{A}} $: a vector $ \mathbf{s}=(s_{\alpha})_{\alpha\in \mathcal{A}} $ corresponds to $ \sum_{\alpha\in \mathcal{A}}s_{\alpha}D_{\alpha} $.
\end{nota}

\section{Height zeta functions and representation theory of the Heisenberg group}
In this section we recall some basic properties of Height zeta functions and review the representation theory of the  Heisenberg group in the adelic setting. Let $ G $ be the Heisenberg group over $ \mathbb{Q} $ and let $ X $ be a smooth projective biequivariant compactification of $ G $ defined over $ \mathbb{Q} $. This section is based on Shalika-Tschinkel \cite{ST04}.
\subsection{Height functions}\label{height pairing}
Let us consider the decomposition of the boundary into irreducible components:
\begin{equation*}
D=X\backslash G=\bigcup_{\alpha\in \mathcal{A}}D_{\alpha}.
\end{equation*}
We fix a smooth adelic metrization $ (\Vert \cdot \Vert_{v})_{v\in \Val({\mathbb{Q}})} $ for each line bundle $ \mathcal{O}(D_{\alpha}) $. Let $ \mathsf{f}_{\alpha} $ be a global section of $ \mathcal{O}(D_{\alpha}) $ corresponding to $ D_{\alpha} $. For each place $ v $, the local height pairing is defined by
\begin{equation*}
\mathsf{H}_{v}: G({\mathbb{Q}}_{v})\times \Pic (X)_{\mathbb{C}}\rightarrow \mathbb{C}^{\times},\;\left(g_{v}, \sum_{\alpha\in \mathcal{A}}s_{\alpha}D_{\alpha}\right)\mapsto \prod_{\alpha\in \mathcal{A}}\Vert \mathsf{f}_{\alpha}(g_{v}) \Vert_{v}^{-s_{\alpha}},
\end{equation*}
and the global height pairing is
\begin{equation*}
\mathsf{H}=\prod_{v\in \Val({\mathbb{Q}})}\mathsf{H}_{v}: G({\mathbb{A}_{\mathbb{Q}}})\times \Pic (X)_{\mathbb{C}}\rightarrow \mathbb{C}^{\times}.
\end{equation*}
We denote by $ \Val(\mathbb{Q})_{\mathrm{fin}} $ the set of all nonarchimedean places of $ \mathbb{Q} $ and by $ \mathbb{A}_{\mathrm{fin}} $ the restricted product $ \prod'_{v}\mathbb{Q}_{v} $ with respect to $ \mathbb{Z}_{v} $ where $ v\in\Val(\mathbb{Q})_{\mathrm{fin}}  $ and $ \mathbb{Q}_{v} $ is the completion of $ \mathbb{Q} $ at $ v $.
\begin{lemm}[{\cite[Proposition 2.3]{ST04}},{\cite[Lemma 6.1]{PSTVA19}}]\label{lemma good reduction}
With the notations above, the height pairing satisfies the following relation:
\begin{equation*}
\mathsf{H}(g,\mathbf{s}+\mathbf{s}')=\mathsf{H}(g,\mathbf{s})\mathsf{H}(g,\mathbf{s}')
\end{equation*}
for all $ \mathbf{s},\mathbf{s}'\in \Pic (X)_{\mathbb{C}}$, all $ g\in G({\mathbb{A}_{\mathbb{Q}}}) $ and there is a compact open subgroup 
\begin{equation*}
\mathsf{K}=\prod_{v}\mathsf{K}_{v}\subset G(\mathbb{A}_{\mathrm{fin}})
\end{equation*}
such that for all $ v\in \Val(\mathbb{Q})_{\mathrm{fin}} $, one has $ \mathsf{H}_{v}(k_{v}g_{v}k'_{v},\mathbf{s})=\mathsf{H}_{v}(g_{v},\mathbf{s}) $ for all $ \mathbf{s}\in \Pic (X)_{\mathbb{C}}$, $ k_{v},k'_{v}\in \mathsf{K}_{v} $ and $ g_{v}\in G({\mathbb{Q}}_{v}) $.\par 
Moreover, if
\begin{enumerate}
\item there is a smooth projective $ \mathbb{Z}_{v} $-model $ \mathcal{X} $ for $ X $ which comes equipped with both left and right actions of the $ \mathbb{Z}_{v} $-group scheme $ G_{\mathbb{Z}_{v}} $ extending the given actions of $ G $ on $ X $; 
\item the metric $ \Vert \cdot \Vert_{v} $ is induced by the integral model $ (\mathcal{X}, \mathcal{D}) $;
\item the unique linearisation on $ \mathcal{O}(D_{\alpha}) $ extends to $ \mathcal{O}(\mathcal{D}_{\alpha}) $ for every $ \alpha\in \mathcal{A} $,
\end{enumerate}
then we can take $ \mathsf{K}_{v}=G(\mathbb{Z}_{v}) $.\par
In particular we can take
\begin{equation*}
\mathsf{K}=\prod_{p\not \in S'}G(\mathbb{Z}_{p})\cdot \prod_{p\in S'}G(p^{n_{p}}\mathbb{Z}_{p})
\end{equation*}
where $ S' $ is a finite set of primes and $ n_{p} $ are positive integers.
\end{lemm}
The height zeta function of $ G(\mathbb{Q})_{\varepsilon} $ associated to the height pairing is defined to be 
\begin{equation}
\mathsf{Z}_{\varepsilon}(\mathbf{s},g)=\sum_{\gamma\in G(\mathbb{Q})}\delta_{\varepsilon}(\gamma g)\mathsf{H}(\mathbf{s},\gamma g)^{-1}.
\end{equation}
where $ g\in G(\mathbb{A}_{\mathbb{Q}}) $. As we stated in the Introduction we shall use the following restriction of $ \mathsf{Z}_{\varepsilon}(\mathbf{s},g) $:
\begin{equation}
\mathsf{Z}_{\varepsilon}(\mathbf{s})=\sum_{\gamma\in G(\mathbb{Q})}\delta_{\varepsilon}(\gamma)\mathsf{H}(\mathbf{s},\gamma )^{-1}.
\end{equation}

\subsection{Representation theory of the Heisenberg group}\label{rep of heisen}
Here we recall the representation theory of the Heisenberg group $ G $ from Shalika and Tschinkel \cite[\S 3]{ST04}.\par
First of all we introduce some necessary notations. We denote by $ \mathrm{Z} $ the center of $ G $ and by $ G^{ab}=G/ \mathrm{Z} $ the abelianization of $ G $. We define the subgroups
\begin{equation*}
\mathrm{U}:=\{u\in G\vert u=g(0,z,y)\},
\end{equation*}
and
\begin{equation*}
\mathrm{W}:=\{w\in G\vert w=g(x,0,0)\}.
\end{equation*}
For the compact open subgroup
\begin{equation*}
\mathsf{K}:=\prod_{p\not \in S'}G(\mathbb{Z}_{p})\cdot \prod_{p\in S'}G(p^{n_{p}}\mathbb{Z}_{p}),
\end{equation*}
we put
\begin{equation}\label{nk}
n(\mathsf{K})=\prod_{p\in S'}p^{n_{p}},
\end{equation}
and denote
\begin{equation*}
\mathsf{K}_{\mathrm{Z}}:=\mathsf{K}\cap \mathrm{Z},
\end{equation*}
\begin{equation*}
\mathsf{K}^{ab}:=\mathsf{K}/\mathsf{K}_{\mathrm{Z}}.
\end{equation*}
Let $ v $ be a place of $ \mathbb{Q} $, through the paper we define the local zeta function by
\begin{equation*}
\zeta_{\mathbb{Q}_{v}}(s)=\begin{cases}
s^{-1}& \text{if}\: \mathbb{Q}_{v}=\mathbb{R},\\ \dfrac{1}{1-p^{-s}} & \text{if}\: v=p\: \text{is nonarchimedean}.
\end{cases} 
\end{equation*}
\begin{nota}
Through the paper, for $ \mathbf{s}=(s_{1},\cdots ,s_{n})\in \mathbb{C}^{n} $ and $ c\in \mathbb{R} $, by $ \Re(\mathbf{s})>c $ we mean that $ \Re(s_{i})>c $ for all $ i\in \{1,\cdots ,n\} $. We denote $ \Vert\mathbf{s}\Vert:=\sqrt{\vert s_{1}\vert^{2}+\cdots +\vert s_{n}\vert^{2}} $. For $ c\in \mathbb{R} $ we denote the tube domain
\begin{equation}
\mathsf{T}_{>c}=\{\mathbf{s}\in \Pic(X)_{\mathbb{C}}: \Re(s_{\alpha})>\kappa_{\alpha}-\varepsilon_{\alpha}+c,  \alpha\in \mathcal{A} \}.
\end{equation}
\end{nota}
We denote the Haar measure on $ G(\mathbb{A}_{\mathbb{Q}}) $ by $ \mathrm{d}g=\prod_{p}\mathrm{d}g_{p}\cdot \mathrm{d}g_{\infty} $ where $ \mathrm{d}g_{\infty}=\mathrm{d}x_{\infty}\mathrm{d}y_{\infty}\mathrm{d}z_{\infty} $  and $ \mathrm{d}g_{p}=\mathrm{d}x_{p}\mathrm{d}y_{p}\mathrm{d}z_{p} $ with the normalizations $ \int_{\mathbb{Z}_{p}}\mathrm{d}x_{p}=1, \int_{\mathbb{Z}_{p}}\mathrm{d}y_{p}=1, \int_{\mathbb{Z}_{p}}\mathrm{d}z_{p}=1 $ and $ \mathrm{d}x_{\infty}, \mathrm{d}y_{\infty}, \mathrm{d}z_{\infty} $ are usual Lebesgue measures on $ \mathbb{R} $. We write $ \mathrm{d}u_{p}=\mathrm{d}y_{p}\mathrm{d}z_{p} $ (resp. $ \mathrm{d}u_{\infty},\mathrm{d}u) $ for the Haar measure on $ \mathrm{U}(\mathbb{Q}_{p}) $ (resp. $ \mathrm{U}(\mathbb{R}), \mathrm{U}(\mathbb{A}_{\mathbb{Q}}) $) and $ \mathrm{d}k_{p} $ for the Haar measure on $ \mathsf{K}_{p} $ with the normalization $ \int_{\mathsf{K}_{p}}\mathrm{d}k_{p}=1 $.\par
Let $ \varrho $ be the right regular representation of $ G(\mathbb{A}_{\mathbb{Q}}) $ on the Hilbert space \cite{Howe}
\begin{equation*}
\mathcal{H}:=\mathsf{L}^{2}\left(G(\mathbb{Q})\backslash G(\mathbb{A}_{\mathbb{Q}})\right).
\end{equation*}
By Peter-Weyl theorem there is a decomposition
\begin{equation*}
\mathcal{H}=\oplus \mathcal{H}_{\psi}
\end{equation*}
where
\begin{equation*}
\mathcal{H}_{\psi}=\{\varphi\in \mathcal{H}|\varrho(z)(\varphi)(g)=\psi(z)\varphi(g)\:\text{for all}\: z\in\mathrm{Z}(\mathbb{A}_{\mathbb{Q}}),g\in G(\mathbb{A}_{\mathbb{Q}})\},
\end{equation*}
and $ \psi $ runs over the set of unitary characters of $ \mathrm{Z}(\mathbb{A}_{\mathbb{Q}}) $ which are trivial on $ \mathrm{Z}(\mathbb{Q}) $. For nontrivial $ \psi $, the corresponding representation $ (\varrho_{\psi},\mathcal{H}_{\psi}) $ of $ G(\mathbb{A}_{\mathbb{Q}}) $ is nontrivial, irreducible and unitary. When $ \psi $ is the trivial character, the corresponding representation $ \varrho_{0} $ decomposes  as a direct sum of one dimensional representations $ \varrho_{\eta} $
\begin{equation*}
\mathcal{H}_{0}=\oplus_{\eta} \mathcal{H}_{\eta}
\end{equation*}
where $ \eta $ runs over all unitary characters of the group $ G^{ab}(\mathbb{Q})\backslash G^{ab}(\mathbb{A}_{\mathbb{Q}}) $.\par
We consider $ \eta $ as a function on $ G(\mathbb{A}_{\mathbb{Q}}) $, trivial on the $ \mathrm{Z}(\mathbb{A}_{\mathbb{Q}}) $ cosets as follows. Let $ \psi_{1}=\prod_{p}\psi_{1,p}\cdot \psi_{1,\infty} $ be the Tate character \cite{Tate}, and for $ \mathbf{a}=(a_{1},a_{2})\in \mathbb{Q}\oplus \mathbb{Q} $, consider
the corresponding linear form on
\begin{equation*}
G^{ab}(\mathbb{A}_{\mathbb{Q}})=\mathbb{A}_{\mathbb{Q}}\oplus \mathbb{A}_{\mathbb{Q}} 
\end{equation*}
given by
\begin{equation*}
g(x,z,y)\mapsto a_{1}x+a_{2}y.
\end{equation*}
We denote by $ \eta_{\mathbf{a}} $ the corresponding adelic character
\begin{equation*}
\eta_{\mathbf{a}}: g(x,z,y)\mapsto \psi_{1}(a_{1}x+a_{2}y)
\end{equation*}
of $ G(\mathbb{A}_{\mathbb{Q}}) $. For $ a\in \mathbb{Q} $, we denote by $ \psi_{a} $ the adelic character of $ \mathrm{Z}(\mathbb{A}_{\mathbb{Q}}) $ given by
\begin{equation*}
z\mapsto \psi_{1}(az).
\end{equation*}

\begin{lemm}\label{lemma decomposition}
There is a constant $ \delta>0 $ such that for all $ \mathbf{s}\in \mathsf{T}_{>\delta}$, one has
\begin{equation}
\mathsf{Z}_{\varepsilon}(\mathbf{s})=\mathsf{Z}_{0,\varepsilon}(\mathbf{s})+\mathsf{Z}_{1,\varepsilon}(\mathbf{s})+\mathsf{Z}_{2,\varepsilon}(\mathbf{s}),
\end{equation}
where 
\begin{equation}
\mathsf{Z}_{0,\varepsilon}(\mathbf{s})=\int_{G(\mathbb{A}_{\mathbb{Q}})}\mathsf{H}(\mathbf{s},g)^{-1}\delta_{\varepsilon}(g)\mathrm{d}g,
\end{equation}
\begin{equation}
\mathsf{Z}_{1,\varepsilon}(\mathbf{s})=\sum_{\eta_{\mathbf{a}}}\int_{G(\mathbb{A}_{\mathbb{Q}})}\mathsf{H}(\mathbf{s},g)^{-1}\overline{\eta_{\mathbf{a}}}(g)\delta_{\varepsilon}(g)\mathrm{d}g,
\end{equation}
\begin{equation}
\mathsf{Z}_{2,\varepsilon}(\mathbf{s})=\sum_{\psi_{a}}\sum_{\omega^{\psi_{a}}\in \mathcal{B}(\varrho_{\psi_{a}})}\omega^{\psi_{a}}(e)\int_{G(\mathbb{A}_{\mathbb{Q}})}\mathsf{H}(\mathbf{s},g)^{-1}\overline{\omega}^{\psi_{a}}(g)\delta_{\varepsilon}(g)\mathrm{d}g.
\end{equation}
Here $ \eta_{\mathbf{a}} $ ranges over all nontrivial characters of 
\begin{equation*}
G^{ab}(\mathbb{Q})\cdot \mathsf{K}^{ab}\backslash G^{ab}(\mathbb{A}_{\mathbb{Q}}),
\end{equation*}
$ \psi_{a} $ ranges over all nontrivial characters of 
\begin{equation*}
\mathrm{Z}(\mathbb{Q})\cdot \mathsf{K}_\mathrm{Z}\backslash \mathrm{Z}(\mathbb{A}_{\mathbb{Q}}),
\end{equation*}
$ \overline{\eta_{\mathbf{a}}} $ and $ \overline{\omega} $ are complex conjugates of $ \eta_{\mathbf{a}} $ and $ \omega $ respectively, the set $ \mathcal{B}(\varrho_{\psi_{a}}) $ is a fixed complete orthonormal basis of $ \mathcal{H}_{\psi_{a}}^{\mathsf{K}} $, the $ \mathsf{K} $-invariant subspace of $ \mathcal{H}_{\psi_{a}} $. For details on $ \mathcal{B}(\varrho_{\psi_{a}}) $ see \cite[Lemma 3.13]{ST04}.
\end{lemm}
\begin{proof}
Recall that the height zeta function is 
\begin{equation*}
\mathsf{Z}_{\varepsilon}(\mathbf{s},h)=\sum_{\gamma\in G(\mathbb{Q})}\delta_{\varepsilon}(\gamma h)\mathsf{H}(\mathbf{s},\gamma h)^{-1}.
\end{equation*}
It follows from the proof of \cite[Proposition 3.3]{ST04} that there is a spectral decomposition in the space $ \mathsf{L}^{2}\left(G(\mathbb{Q})\backslash G(\mathbb{A}_{\mathbb{Q}})\right) $:
\begin{equation*}
\mathsf{Z}_{\varepsilon}(\mathbf{s},h)=\mathsf{Z}_{0,\varepsilon}(\mathbf{s},h)+\mathsf{Z}_{1,\varepsilon}(\mathbf{s},h)+\mathsf{Z}_{2,\varepsilon}(\mathbf{s},h),
\end{equation*}
where 
\begin{equation*}
\mathsf{Z}_{0,\varepsilon}(\mathbf{s},h)=\int_{G(\mathbb{A}_{\mathbb{Q}})}\mathsf{H}(\mathbf{s},g)^{-1}\delta_{\varepsilon}(g)\mathrm{d}g,
\end{equation*}
\begin{equation}\label{value of eta}
\mathsf{Z}_{1,\varepsilon}(\mathbf{s},h)=\sum_{\eta_{\mathbf{a}}}\eta_{\mathbf{a}}(h)\int_{G(\mathbb{A}_{\mathbb{Q}})}\mathsf{H}(\mathbf{s},g)^{-1}\overline{\eta_{\mathbf{a}}}(g)\delta_{\varepsilon}(g)\mathrm{d}g,
\end{equation}
\begin{equation*}
\mathsf{Z}_{2,\varepsilon}(\mathbf{s},h)=\sum_{\psi_{a}}\sum_{\omega^{\psi_{a}}\in\mathcal{B}(\varrho_{\psi})}\omega^{\psi_{a}}(h)\int_{G(\mathbb{A}_{\mathbb{Q}})}\mathsf{H}(\mathbf{s},g)^{-1}\overline{\omega}^{\psi_{a}}(g)\delta_{\varepsilon}(g)\mathrm{d}g.
\end{equation*}
Now we restrict $ \mathsf{Z}_{\varepsilon}(\mathbf{s},h) $ to the identity $ h=e $ and the resulting function is 
\begin{equation*}
\mathsf{Z}_{\varepsilon}(\mathbf{s})=\sum_{\gamma\in G(\mathbb{Q})}\delta_{\varepsilon}(\gamma)\mathsf{H}(\mathbf{s},\gamma )^{-1}.
\end{equation*}
In (\ref{value of eta}), we have $ \eta_{\mathbf{a}}(h)=\eta_{\mathbf{a}}(e)=1 $. Thus the lemma follows.
\end{proof}
\begin{nota}\label{eta and a}
We denote by $ \Lambda_{1} $ the lattice of $ G^{ab}(\mathbb{Q}) $ consisting of  $ \mathbf{a} $ such that $ \eta_{\mathbf{a}} $ is trivial on $ \mathsf{K}^{ab} $. We also denote by $ \Lambda_{2} $ the lattice of $ \mathrm{Z}(\mathbb{Q}) $ consisting of $ a $ such that $ \psi_{a} $ is trivial on $ \mathsf{K}_\mathrm{Z} $.
\end{nota}
\begin{lemm}\label{k-invariant}
The height function $ \mathsf{H}(\mathbf{s},g) $ and $ \delta_{\varepsilon} $ are bi-$ \mathsf{K} $-invariant, that is,
\begin{equation*}
\mathsf{H}(\mathbf{s},gk)=\mathsf{H}(\mathbf{s},kg)=\mathsf{H}(\mathbf{s},g),
\end{equation*}
\begin{equation*}
\delta_{\varepsilon}(gk)=\delta_{\varepsilon}(kg)=\delta_{\varepsilon}(g),
\end{equation*}
for all $ g\in G $ and $ k\in \mathsf{K} $.
\end{lemm}
\begin{proof}
This follows from \cite{ST04} and \cite[Lemma 6.2]{PSTVA19}.
\end{proof}

We are going to compute
\begin{equation*}
\mathsf{Z}_{\varepsilon}(\mathbf{s})=\mathsf{Z}_{0,\varepsilon}(\mathbf{s})+\mathsf{Z}_{1,\varepsilon}(\mathbf{s})+\mathsf{Z}_{2,\varepsilon}(\mathbf{s}).
\end{equation*}
Recall that $ S $ is a finite set of places of $ \mathbb{Q} $ containing the archimedean place such that there is a good integral model $ (\mathcal{X},\mathcal{D}) $  for $ (X,D) $ over the ring of $ S $-integers $ \mathbb{Z}_S $, and hence $ (\mathcal{X},\mathcal{D}_{\varepsilon}) $ is a good integral model for $ (X,D_{\varepsilon}) $ over  $ \mathbb{Z}_S $. We are going to count the Campana $ \bZ_{S} $-points on $ (\mathcal{X}, \cD_{\varepsilon}) $.

\section{Height integrals I}\label{sec1}
In this section, we study the height integral
\begin{equation*}
\mathsf{Z}_{0,\varepsilon}(\mathbf{s})=\int_{G(\mathbb{A}_{\mathbb{Q}})}\mathsf{H}(\mathbf{s},g)^{-1}\delta_{\varepsilon}(g)\mathrm{d}g. 
\end{equation*}
Our analysis here is similar to that of Pieropan et al.~ \cite{PSTVA19}. Let us first set up some notation. For all places $ v $ of $ \mathbb{Q} $, write
\begin{equation*}
D\otimes_{\mathbb{Q}}\mathbb{Q}_{v}=\bigcup_{\beta\in \mathcal{A}_{v}}D_{v,\beta},
\end{equation*}
\begin{equation*}
D_{\alpha}\otimes_{\mathbb{Q}}\mathbb{Q}_{v}=\bigcup_{\beta\in \mathcal{A}_{v}(\alpha)}D_{v,\beta},
\end{equation*}
for the decompositions into irreducible components.\par
Let $ \beta\in \mathcal{A}_{v} $. We denote by $ \mathbb{Q}_{v,\beta} $ the field of definition for one of the geometric irreducible components of $ D_{v,\beta} $  and denote the extension degree $ [\mathbb{Q}_{v,\beta}:\mathbb{Q}_{v}] $ by $ f_{v,\beta} $ .\par
For any subset $ B\subseteq \mathcal{A}_{v} $, define
\begin{equation*}
D_{v,B}:=\bigcap_{\beta\in B}D_{v,\beta},\qquad D_{v,B}^{\circ}:=D_{v,B}\backslash \bigcup_{B\subsetneq B'\subset \mathcal{A}_{v}}\left(\bigcap_{\beta\in B'}D_{v,\beta}\right)
\end{equation*}
where we set that $ D_{v,\varnothing}=X_{\mathbb{Q}_{v}} $ and $ D_{v,\varnothing}^{\circ}=G_{\mathbb{Q}_{v}} $. For $ v\not \in S $, we denote by $ \mathcal{D}_{v,B} $ the Zariski closure of $D_{v,B} $ in $ \mathcal{X}\otimes_{\mathbb{Z}_{S}}\mathbb{Z}_{v} $. We define $ \mathcal{D}_{v,B}^{\circ} $ analogously.\par 
We denote the restricted product $ \prod'_{v\not\in S}\mathbb{Q}_{v} $ by $ \mathbb{A}^{S} $.\par 
Recall that we fix $ S $ being a finite set of places of $ F $ containing all archimedean places and recall from the definition of Campana points that
\begin{equation*}
m_{\alpha}=\dfrac{1}{1-\varepsilon_{\alpha}}.
\end{equation*}
Also recall that we write the anticanonical divisor of $ X $ as $ -K_{X}=\sum_{\alpha\in \mathcal{A}}\kappa_{\alpha}D_{\alpha} $. 
\subsection{Places not in $ S $}
Assume that $ p\not\in S $. We shall consider the integral
\begin{equation}
\int_{G(\mathbb{A}^{S})}\mathsf{H}(\mathbf{s},g)^{-1}\delta_{\varepsilon}(g)\mathrm{d}g=\prod_{p\not\in S}\int_{G(\mathbb{Q}_{p})}\mathsf{H}_{p}(\mathbf{s},g_{p})^{-1}\delta_{\varepsilon,p}(g_{p})\mathrm{d}g_{p},
\end{equation}
and we denote
\begin{equation*}
\mathsf{Z}_{0,\varepsilon,p}(\mathbf{s}):=\int_{G(\mathbb{Q}_{p})}\mathsf{H}_{p}(\mathbf{s},g_{p})^{-1}\delta_{\varepsilon,p}(g_{p})\mathrm{d}g_{p}.
\end{equation*}
\subsubsection{Places of good reduction}\label{sec2}
We assume the integral model $ \mathcal{X} $ has good reduction at $ p $ and the metrics at $ p $ are induced by our integral model. 
\begin{prop}\label{alphabeta}
we have
\begin{equation}\label{eq2.2}
\mathsf{Z}_{0,\varepsilon,p}(\mathbf{s})=\sum_{B\subset \mathcal{A}_{p}}\dfrac{\#\mathcal{D}_{p,B}^{\circ}(\mathbb{F}_{p})}{p^{3-\# B}}\prod_{\beta\in B}\left(1-\dfrac{1}{p}\right)\dfrac{p^{-m_{\alpha(\beta)}(s_{\alpha(\beta)}-\kappa_{\alpha(\beta)}+1)}}{1-p^{-(s_{\alpha(\beta)}-\kappa_{\alpha(\beta)}+1)}}
\end{equation}
where we interpret the term $ p^{-m_{\alpha(\beta)}(s_{\alpha(\beta)}-\kappa_{\alpha(\beta)}+1)}=0 $ if $ \varepsilon_{\alpha(\beta)}=1 $.
\end{prop}
\begin{proof}
The proof is analogous to that of \cite[Theorem 7.1]{PSTVA19}. We set $ \mathbf{\kappa}=(\kappa_{\alpha})_{\alpha\in \mathcal{A}} $ and let $ \mathrm{d}\tau $ denote the Tamagawa measure. Consider the reduction map
\begin{equation*}
\rho_{p}: G(\mathbb{Q}_{p})\rightarrow \mathcal{X}(\mathbb{F}_{p}),
\end{equation*}
then by stratification we have
\begin{equation*}
\mathsf{Z}_{0,\varepsilon,p}(\mathbf{s})=\sum_{B\subset \mathcal{A}_{p}}\sum_{y\in \mathcal{D}_{p,B}^{\circ}(\mathbb{F}_{p})}\int_{\rho_{p}^{-1}(y)}\mathsf{H}_{p}(\mathbf{s}-\mathbf{\kappa},g_{p})^{-1}\delta_{\varepsilon,p}(g_{p})\mathrm{d}\tau .
\end{equation*}
We now compute the inner integral
\begin{equation}\label{eqC}
\int_{\rho_{p}^{-1}(y)}\mathsf{H}_{p}(\mathbf{s}-\mathbf{\kappa},g_{p})^{-1}\delta_{\varepsilon,p}(g_{p})\mathrm{d}\tau.
\end{equation}
If $ B=\varnothing $, then we have
\begin{equation*}
\mathsf{H}_{p}(\mathbf{s}-\mathbf{\kappa},g_{p})=\delta_{\varepsilon,p}(g_{p})=1,
\end{equation*}
so (\ref{eqC}) evaluates to $ \frac{1}{p^{3}} $.\par
If $ B\neq \varnothing $, we write $ B=\{\beta_{1},\cdots ,\beta_{\ell}\}(\ell\leq 3) $ and every $ \beta_{i} $ lies above a unique $ \alpha_{i}:=\alpha(\beta_{i})\in \mathcal{A} $. In a neighborhood  of $ \rho_{p}^{-1}(y) $,  there exist analytic local coordinates $ (z_{1},z_{2},z_{3}) $ on $ \rho_{p}^{-1}(y) $ inducing a measure-preserving analytic isomorphism $ \rho_{p}^{-1}(y)\cong \mathfrak{m}_{p}^{3} $. The integral (\ref{eqC}) then can be rewritten as
\begin{equation*}
\int_{\rho_{p}^{-1}(y)}\mathsf{H}_{p}(\mathbf{s}-\mathbf{\kappa},g_{p})^{-1}\delta_{\varepsilon,p}(g_{p})\mathrm{d}\tau  =\int_{\mathfrak{m}_{p}^{3}}\prod_{i=1}^{\ell}(|z_{i}|_{p}^{s_{\alpha_{i}}-\kappa_{\alpha_{i}}}\delta_{\varepsilon,p}(z_{i})\mathrm{d}z_{i})\prod_{i>\ell}\mathrm{d}z_{i}
\end{equation*}
where by definition $ \delta_{\varepsilon,p}(z_{i})=1 $ is equivalent to
\begin{equation*}
\varepsilon_{\alpha_{i}}\neq 1 \quad \mathrm{and}\quad v_{p}(z_{i})\geq m_{i}=\dfrac{1}{1-\varepsilon_{\alpha_{i}}}
\end{equation*}
where $ v_{p}(z_{i}) $ is the $ p $-adic valuation of $ z_{i} $.\par
Thus if $ \Re(s_{\alpha_{i}})-\kappa_{\alpha_{i}}+1>0 $ for all $ i\in \{1,\cdots ,\ell\} $, we have
\begin{eqnarray*}
\int_{\rho_{p}^{-1}(y)}\mathsf{H}_{p}(\mathbf{s}-\mathbf{\kappa},g_{p})^{-1}\delta_{\varepsilon,p}(g_{p})\mathrm{d}\tau  &=&\dfrac{1}{p^{3-\ell}}\prod_{i=1}^{\ell}\sum_{j=m_{i}}^{\infty}p^{-j(s_{\alpha_{i}}-\kappa_{\alpha_{i}})}\Vol(p^{j}\mathbb{Z}_{p}^{\times})\\
&=&\dfrac{1}{p^{3-\ell}}\prod_{i=1}^{\ell}\sum_{j=m_{i}}^{\infty}p^{-j(s_{\alpha_{i}}-\kappa_{\alpha_{i}})}p^{-j}\left(1-\dfrac{1}{p}\right)\\ 
&=&\dfrac{1}{p^{3-\ell}}\prod_{i=1}^{\ell}\left(1-\dfrac{1}{p}\right)\dfrac{p^{-m_{i}(s_{\alpha_{i}}-\kappa_{\alpha_{i}}+1)}}{1-p^{-(s_{\alpha_{i}}-\kappa_{\alpha_{i}}+1)}}.
\end{eqnarray*}
Summing over different subsets of $ \mathcal{A}_{p} $, we have
\begin{equation}
\mathsf{Z}_{0,\varepsilon,p}(\mathbf{s})=\sum_{B\subset \mathcal{A}_{p}}\dfrac{\#\mathcal{D}_{p,B}^{\circ}(\mathbb{F}_{p})}{p^{3-\# B}}\prod_{\beta\in B}\left(1-\dfrac{1}{p}\right)\dfrac{p^{-m_{\alpha(\beta)}(s_{\alpha(\beta)}-\kappa_{\alpha(\beta)}+1)}}{1-p^{-(s_{\alpha(\beta)}-\kappa_{\alpha(\beta)}+1)}},
\end{equation}
where we interpret the term $ p^{-m_{\alpha(\beta)}(s_{\alpha(\beta)}-\kappa_{\alpha(\beta)}+1)}=0 $ if $ \varepsilon_{\alpha(\beta)}=1 $.
\end{proof}

\subsubsection{Places of bad reduction}
Again let $ p\not \in S $. We assume the integral model $ \mathcal{X} $ has bad reduction at $ p $ or the metrics at $ p $ are not induced by our integral model.
\begin{prop}\label{propNN}
The function $ \mathsf{Z}_{0,\varepsilon,p}(\mathbf{s}) $ is holomorphic when $ \Re(s_{\alpha})>\kappa_{\alpha}-1 $ for  $ \alpha\in \mathcal{A} $ with $\varepsilon_{\alpha}<1 $.
\end{prop}
\begin{proof}
We have
\begin{equation*}
\mathsf{Z}_{0,\varepsilon,p}(\mathbf{s})=\int_{G(\mathbb{Q}_{p})}\mathsf{H}_{p}(\mathbf{s},g_{p})^{-1}\delta_{\varepsilon,p}(g_{p})\mathrm{d}g_{p}=\int_{G(\mathbb{Q}_{p})}\prod_{\alpha\in \mathcal{A}}\Vert \mathsf{f}_{\alpha}(g_{p}) \Vert_{p}^{s_{\alpha}}\delta_{\varepsilon,p}(g_{p})\mathrm{d}g_{p}.
\end{equation*}
The proposition then follows from \cite[Lemma 4.1]{volume} by taking $ \Phi=\delta_{\varepsilon,p} $.
\end{proof}

\subsection{Places in $ S $}
Let $ v\in S $. Then $ \delta_{\varepsilon,v}\equiv 1 $ by definition. 
\begin{prop}\label{propN}
\begin{enumerate}
\item The function
\begin{equation*}
\mathsf{Z}_{0,\varepsilon,v}(\mathbf{s}):=\int_{G(\mathbb{Q}_{v})}\mathsf{H}_{v}(\mathbf{s},g_{v})^{-1}\delta_{\varepsilon,v}(g_{v})\mathrm{d}g_{v}
\end{equation*}
is holomorphic when $ \Re(s_{\alpha})>\kappa_{\alpha}-1 $ for all $ \alpha\in \mathcal{A} $.
\item Suppose $ L=\sum_{\alpha\in \mathcal{A}}\lambda_{\alpha}D_{\alpha} $ is a big divisor on $ X $ and let
\begin{equation*}
a:=\tilde{a}((X,D_{\mathrm{red}}),L),\quad b:=b(\mathbb{Q}_{v},(X,D_{\mathrm{red}}),L)
\end{equation*}
be as defined in \cite[\S 4]{PSTVA19}. Then there is a constant $ \delta>0 $ such that the function
\begin{equation*}
s\mapsto \zeta_{\mathbb{Q}_{v}}(s-a)^{-b}\mathsf{Z}_{0,\varepsilon,v}(sL)
\end{equation*}
admits a holomorphic continuation to the domain  $ \Re(s)>a-\delta $ and the function $ s\mapsto \mathsf{Z}_{0,\varepsilon,v}(sL) $ has a pole at $ s=a $ of order $ b $.
\end{enumerate}
\end{prop}
\begin{proof}
\begin{enumerate}
\item The assertion follows from \cite[Lemma 4.1]{ST04}.
\item This follows from  \cite[Proposition 4.3]{volume} by taking $ \Phi=\delta_{\varepsilon,v}\equiv 1 $. See also \cite[Corollary 4.4]{volume}.
\end{enumerate}
\end{proof}

\subsection{Euler products}
Let $ \alpha\in \mathcal{A} $. We denote the field of definition for one of the geometrically irreducible components of $ D_{\alpha} $ by $ F_{\alpha} $.

\begin{prop}\label{propG}
Let $ p\not\in S $ and our integral model $ \mathcal{X} $ has good reduction at $ p $. Recall the decomposition
\begin{equation*}
D_{\alpha}\otimes_{\mathbb{Q}}\mathbb{Q}_{p}=\bigcup_{\beta\in \mathcal{A}_{p}(\alpha)}D_{p,\beta}
\end{equation*}
of $ D_{\alpha}\otimes_{\mathbb{Q}}\mathbb{Q}_{p} $ into irreducible components.
\begin{enumerate}
\item The function
\begin{equation*}
\mathbf{s}\mapsto \prod_{\alpha\in \mathcal{A}}\prod_{\beta\in \mathcal{A}_{p}(\alpha)}\zeta_{\mathbb{Q}_{p,\beta}}(m_{\alpha}(s_{\alpha}-\kappa_{\alpha}+1))^{-1}\mathsf{Z}_{0,\varepsilon,p}(\mathbf{s})
\end{equation*}
is holomorphic on $ \mathsf{T}_{>-\delta} $ for $\delta>0$ sufficiently small. If $ \varepsilon_{\alpha}=1 $, we take $ \zeta_{\mathbb{Q}_{p,\beta}}(m_{\alpha}(s_{\alpha}-\kappa_{\alpha}+1))^{-1}=1 $.
\item Let $\delta>0$ be sufficiently small, then there is a constant $ \delta'>0 $ such that
\begin{equation*}
\prod_{\alpha\in \mathcal{A}}\prod_{\beta\in \mathcal{A}_{p}(\alpha)}\zeta_{\mathbb{Q}_{p,\beta}}(m_{\alpha}(s_{\alpha}-\kappa_{\alpha}+1))^{-1}\mathsf{Z}_{0,\varepsilon,p}(\mathbf{s})=1+O(p^{-(1+\delta')})
\end{equation*}
for $\mathbf{s}\in \mathsf{T}_{>-\delta} $.
\end{enumerate}
\end{prop}
\begin{proof}
The proof is analogous to that of \cite[Proposition 7.4]{PSTVA19}. By Proposition \ref{alphabeta}, we have
\begin{equation*}
\mathsf{Z}_{0,\varepsilon,p}(\mathbf{s})=\sum_{B\subset \mathcal{A}_{p}}\dfrac{\#\mathcal{D}_{p,B}^{\circ}(\mathbb{F}_{p})}{p^{3-\# B}}\prod_{\beta\in B}\left(1-\dfrac{1}{p}\right)\dfrac{p^{-m_{\alpha(\beta)}(s_{\alpha(\beta)}-\kappa_{\alpha(\beta)}+1)}}{1-p^{-(s_{\alpha(\beta)}-\kappa_{\alpha(\beta)}+1)}}.
\end{equation*}
\begin{enumerate}
\item  If $ B=\varnothing $, then $ \#\mathcal{D}_{p,B}^{\circ}(\mathbb{F}_{p})=\#G(\mathbb{F}_{p})=p^{3} $ and thus 
\begin{equation*}
\dfrac{\#\mathcal{D}_{p,B}^{\circ}(\mathbb{F}_{p})}{p^{3-\# B}}\prod_{\beta\in B}\left(1-\dfrac{1}{p}\right)\dfrac{p^{-m_{\alpha(\beta)}(s_{\alpha(\beta)}-\kappa_{\alpha(\beta)}+1)}}{1-p^{-(s_{\alpha(\beta)}-\kappa_{\alpha(\beta)}+1)}}=\dfrac{p^{3}}{p^{3}}=1.
\end{equation*}
\item  If $ B=\{\beta\} $, then $ B $ lies above a unique $ \alpha(\beta)\in \mathcal{A} $. If $ \mathcal{D}_{p,B}^{\circ}(\mathbb{F}_{p})=\varnothing $ or $ \varepsilon_{\alpha(\beta)}=1 $, then
\begin{equation*}
\dfrac{\#\mathcal{D}_{p,B}^{\circ}(\mathbb{F}_{p})}{p^{3-\# B}}\prod_{\beta\in B}\left(1-\dfrac{1}{p}\right)\dfrac{p^{-m_{\alpha(\beta)}(s_{\alpha(\beta)}-\kappa_{\alpha(\beta)}+1)}}{1-p^{-(s_{\alpha(\beta)}-\kappa_{\alpha(\beta)}+1)}}=0.
\end{equation*}
If $ \mathcal{D}_{p,B}^{\circ}(\mathbb{F}_{p})\neq \varnothing $ and $ \varepsilon_{\alpha(\beta)}\neq 1 $, then 
\begin{equation*}
\#\mathcal{D}_{p,B}^{\circ}(\mathbb{F}_{p})=p^{2}+O(p^{2-\delta_{1}})
\end{equation*}
for some $ \delta_{1}>0 $. Therefore when $ \delta_{1} $ is sufficiently small and $\mathbf{s}\in \mathsf{T}_{>-\delta} $, we have
\begin{equation*}
\dfrac{\#\mathcal{D}_{p,B}^{\circ}(\mathbb{F}_{p})}{p^{3-\# B}}\prod_{\beta\in B}\left(1-\dfrac{1}{p}\right)\dfrac{p^{-m_{\alpha(\beta)}(s_{\alpha(\beta)}-\kappa_{\alpha(\beta)}+1)}}{1-p^{-(s_{\alpha(\beta)}-\kappa_{\alpha(\beta)}+1)}}=p^{-m_{\alpha(\beta)}(s_{\alpha(\beta)}-\kappa_{\alpha(\beta)}+1)}\left(1+O(p^{-\delta_{2}})\right)
\end{equation*}
for some $ \delta_{2}>0 $. On the other hand when $\mathbf{s}\in \mathsf{T}_{>-\delta} $, we have
\begin{equation*}
\left| p^{-m_{\alpha(\beta)}(s_{\alpha(\beta)}-\kappa_{\alpha(\beta)}+1)}\right| \leq p^{-(1-m_{\alpha(\beta)}\delta)}.
\end{equation*}
So if $  \delta$ is sufficiently small,
\begin{equation*}
p^{-m_{\alpha(\beta)}(s_{\alpha(\beta)}-\kappa_{\alpha(\beta)}+1)}\left(1+O(p^{-\delta_{2}})\right)=p^{-m_{\alpha(\beta)}(s_{\alpha(\beta)}-\kappa_{\alpha(\beta)}+1)}+O(p^{-(1+\delta')})
\end{equation*}
for some $ \delta'>0 $.
\item  If $ \#B\geq 2 $, then $ \#\mathcal{D}_{p,B}^{\circ}(\mathbb{F}_{p})=O(p^{3-\#B}) $. Moreover when $\mathbf{s}\in \mathsf{T}_{>-\delta} $,
\begin{equation*}
\left| \left(1-\dfrac{1}{p}\right)\dfrac{p^{-m_{\alpha(\beta)}(s_{\alpha(\beta)}-\kappa_{\alpha(\beta)}+1)}}{1-p^{-(s_{\alpha(\beta)}-\kappa_{\alpha(\beta)}+1)}}\right| \leq p^{-(1-m\delta)}
\end{equation*}
for some $ m>0 $. Thus
\begin{equation*}
\left| \prod_{\beta\in B}\left(1-\dfrac{1}{p}\right)\dfrac{p^{-m_{\alpha(\beta)}(s_{\alpha(\beta)}-\kappa_{\alpha(\beta)}+1)}}{1-p^{-(s_{\alpha(\beta)}-\kappa_{\alpha(\beta)}+1)}}\right| \leq p^{-2(1-m\delta)}=O(p^{-(1+\delta')})
\end{equation*}
for $\mathbf{s}\in \mathsf{T}_{>-\delta} $ and sufficiently small $ \delta>0 $, where $ \delta' $ is as above.
\end{enumerate}
From the above analysis we have for $\mathbf{s}\in \mathsf{T}_{>-\delta} $ and sufficiently small $ \delta>0 $,
\begin{equation*}
\mathsf{Z}_{0,\varepsilon,p}(\mathbf{s})=1+\sum_{\alpha\in \mathcal{A}}\sum_{\beta\in \mathcal{A}_{p}(\alpha)\atop f_{p,\beta}=1}p^{-m_{\alpha}(s_{\alpha}-\kappa_{\alpha}+1)}+O(p^{-(1+\delta')})
\end{equation*}
where $ f_{p,\beta}=[\mathbb{Q}_{p,\beta}:\mathbb{Q}_{p}] $ and thus
\begin{equation*}
\prod_{\alpha\in \mathcal{A}}\prod_{\beta\in \mathcal{A}_{p}(\alpha)}\left(1-p^{-f_{p,\beta}m_{\alpha}(s_{\alpha}-\kappa_{\alpha}+1)}\right)\mathsf{Z}_{0,\varepsilon,p}(\mathbf{s})=1+O(p^{-(1+\delta')}),
\end{equation*}
which implies the proposition.
\end{proof}

\begin{coro}\label{propO}
The function
\begin{equation*}
\mathbf{s}\mapsto \left(\prod_{\alpha\in \mathcal{A}}\zeta_{F_{\alpha}}(m_{\alpha}(s_{\alpha}-\kappa_{\alpha}+1))^{-1}\right)\prod_{p\not\in S}\mathsf{Z}_{0,\varepsilon,p}(\mathbf{s})
\end{equation*}
is holomorphic on the domain $ \mathsf{T}_{>-\delta'} $.
\end{coro}
\begin{proof}
This follows from Propositions \ref{propNN}, \ref{propG} and the fact that
\begin{equation*}
F_{\alpha}\otimes_{\mathbb{Q}}\mathbb{Q}_{p}\cong \prod_{\beta\in \mathcal{A}_{p}(\alpha)}\mathbb{Q}_{p,\beta}
\end{equation*}
for all $ \alpha\in \mathcal{A} $.
\end{proof}

\section{Height integrals II}
In this section, we study the height integral from Lemma \ref{lemma decomposition}:
\begin{equation*}
\int_{G(\mathbb{A}_{\mathbb{Q}})}\mathsf{H}(\mathbf{s},g)^{-1}\overline{\eta_{\mathbf{a}}}(g)\delta_{\varepsilon}(g)\mathrm{d}g=\prod_{v\in\Val(\mathbb{Q})}\int_{G(\mathbb{Q}_{v})}\mathsf{H}_{v}(\mathbf{s},g_{v})^{-1}\overline{\eta_{\mathbf{a}}}(g_{v})\delta_{\varepsilon,v}(g_{v})\mathrm{d}g_{v}.
\end{equation*}
where $ \mathbf{a}\in \Lambda_{1} $ (see Notation \ref{eta and a} for $ \Lambda_{1} $). Recall that in \S \ref{rep of heisen}, for $ \mathbf{a}=(a_{1},a_{2})\in \mathbb{Q}\oplus \mathbb{Q} $, we denote by $ \eta_{\mathbf{a}} $ the corresponding automorphic character of $ G(\mathbb{A}_{\mathbb{Q}}) $:
\begin{equation*}
\eta_{\mathbf{a}}: g(x,z,y)\mapsto \psi_{1}(a_{1}x+a_{2}y),
\end{equation*}
where $ \psi_{1} $ is the Tate character \cite{Tate}.\par
The analysis in this section is analogous to that of \cite[\S 8]{PSTVA19}. For each nonzero $ \mathbf{a}=(a_{1}, a_{2})\in \mathbb{Q}\oplus \mathbb{Q} $, we denote the linear functional
\begin{equation*}
(x,y)\mapsto a_{1}x+a_{2}y
\end{equation*}
by $ f_{\mathbf{a}} $. Recall from Corollary \ref{coro divisor} that
\begin{equation*}
\mathrm{div}(f_{\mathbf{a}})=E(f_{\mathbf{a}})-\sum_{\alpha\in \mathcal{A}}d_{\alpha}(f_{\mathbf{a}})D_{\alpha}
\end{equation*}
with $ d_{\alpha}(f_{\mathbf{a}})\geq 0 $ for all $ \alpha\in \mathcal{A} $ and at least one $ \alpha $ such that $ d_{\alpha}(f_{\mathbf{a}})> 0 $. Define
\begin{equation*}
\mathcal{A}^{0}(\mathbf{a}):=\{\alpha\in \mathcal{A}|d_{\alpha}(f_{\mathbf{a}})=0\},
\end{equation*}
\begin{equation*}
\mathcal{A}^{\geq 1}(\mathbf{a}):=\{\alpha\in \mathcal{A}|d_{\alpha}(f_{\mathbf{a}})\geq 1\}.
\end{equation*}
For every place $ v\in \Val(\mathbb{Q}) $ we define
\begin{equation*}
H_{v}(\mathbf{a})=\max\{|a_{1}|_{v},|a_{2}|_{v}\}
\end{equation*}
and for every nonarchimedean place $ p $ define
\begin{equation*}
j_{p}(\mathbf{a})=\min\{v_{p}(a_{1}),v_{p}(a_{2})\}
\end{equation*}
where $ v_{p}(a_{1}),v_{p}(a_{2}) $ are $ p $-adic valuations of $ a_{1} $ and $ a_{2} $ respectively. We denote
\begin{equation*}
H_{\mathrm{fin}}(\mathbf{a})=\prod_{p\in\Val(\mathbb{Q})_{\mathrm{fin}}}H_{p}(\mathbf{a}).
\end{equation*}
Let $ \infty $ denote the infinite place of $ \mathbb{Q} $. Then we have
\begin{equation}\label{eqD}
H_{\infty}(\mathbf{a})\gg H_{\mathrm{fin}}(\mathbf{a})^{-1}.
\end{equation}
We denote
\begin{equation*}
\mathsf{Z}_{1,\varepsilon,v}(\mathbf{s},\eta_{\mathbf{a}}):=\int_{G(\mathbb{Q}_{v})}\mathsf{H}_{v}(\mathbf{s},g_{v})^{-1}\overline{\eta_{\mathbf{a}}}(g_{v})\delta_{\varepsilon,v}(g_{v})\mathrm{d}g_{v}.
\end{equation*}
\subsection{Places not in $ S $}
Let $ p\not \in S $.
\subsubsection{Places of good reduction}
We assume the integral model $ \mathcal{X} $ has good reduction at $ p $ and the metrics at $ p $ are induced by our integral model. There are two cases, $ j_{p}(\mathbf{a})=0 $ or $ j_{p}(\mathbf{a})\neq 0 $. We first assume $ j_{p}(\mathbf{a})=0 $. 
\begin{prop}\label{propH}
\begin{enumerate}
\item There exists $\delta>0$, independent of $ \mathbf{a} $ such that the function
\begin{equation*}
\mathbf{s}\mapsto \prod_{\alpha\in \mathcal{A}^{0}(\mathbf{a})}\prod_{\beta\in \mathcal{A}_{p}(\alpha)}\zeta_{\mathbb{Q}_{p,\beta}}(m_{\alpha}(s_{\alpha}-\kappa_{\alpha}+1))^{-1}\mathsf{Z}_{1,\varepsilon,p}(\mathbf{s},\eta_{\mathbf{a}})
\end{equation*}
is holomorphic on $ \mathsf{T}_{>-\delta} $. If $ \varepsilon_{\alpha}=1 $, we take $ \zeta_{\mathbb{Q}_{p,\beta}}(m_{\alpha}(s_{\alpha}-\kappa_{\alpha}+1))^{-1}=1 $.
\item There exists $\delta'>0$, independent of $ \mathbf{a} $ such that
\begin{equation*}
\prod_{\alpha\in \mathcal{A}^{0}(\mathbf{a})}\prod_{\beta\in \mathcal{A}_{p}(\alpha)}\zeta_{\mathbb{Q}_{p,\beta}}(m_{\alpha}(s_{\alpha}-\kappa_{\alpha}+1))^{-1}\mathsf{Z}_{1,\varepsilon,p}(\mathbf{s},\eta_{\mathbf{a}})=1+O(p^{-(1+\delta')}),
\end{equation*}
for $\mathbf{s}\in \mathsf{T}_{>-\delta} $.
\end{enumerate}
\end{prop}
\begin{proof}
The proof is analogous to that of \cite[Proposition 8.1]{PSTVA19}. As before by stratification we have
\begin{equation*}
\mathsf{Z}_{1,\varepsilon,p}(\mathbf{s},\eta_{\mathbf{a}})=\sum_{B\subset \mathcal{A}_{p}}\sum_{y\in \mathcal{D}_{p,B}^{\circ}(\mathbb{F}_{p})}\int_{\rho_{p}^{-1}(y)}\mathsf{H}_{p}(\mathbf{s}-\mathbf{\kappa},g_{p})^{-1}\delta_{\varepsilon,p}(g_{p})\overline{\eta_{\mathbf{a}}}(g_{p})\mathrm{d}\tau .
\end{equation*}
\begin{enumerate}
\item If $ B=\varnothing $, then the inner sum is 1.
\item If $ B=\{\beta\} $, then $ \beta $ lies above a unique $ \alpha(\beta)\in \mathcal{A} $.  There are two cases, $ \alpha(\beta)\in \mathcal{A}^{0}(\mathbf{a}) $ or $ \alpha(\beta)\in \mathcal{A}^{\geq 1}(\mathbf{a}) $.\par 
If $ \alpha(\beta)\in \mathcal{A}^{0}(\mathbf{a}) $, then the character $ \overline{\eta_{\mathbf{a}}} $ is trivial on $ \rho_{p}^{-1}(\mathcal{D}_{p,B}^{\circ}(\mathbb{F}_{p})) $, the situation is same as that in Proposition \ref{propG}. Thus we have for sufficiently small $ \delta>0 $, the inner summation is
\begin{equation*}
p^{-m_{\alpha(\beta)}(s_{\alpha(\beta)}-\kappa_{\alpha(\beta)}+1)}\left(1+O(p^{-\delta_{1}})\right)
\end{equation*}
for some $ \delta_{1}>0 $.\par
If $ \alpha(\beta)\in \mathcal{A}^{\geq 1}(\mathbf{a}) $, we set $ d:=d_{\alpha(\beta)}(f_{\mathbf{a}}) $, there are two cases. If $ y\not \in E(f_{\mathbf{a}})(\mathbb{F}_{p}) $, then by \cite[Lemma 5.4]{PSTVA19}, we have for sufficiently small $ \delta>0 $,
\begin{equation*}
\int_{\rho_{p}^{-1}(y)}\mathsf{H}_{p}(\mathbf{s}-\mathbf{\kappa},g_{p})^{-1}\delta_{\varepsilon,p}(g_{p})\overline{\eta_{\mathbf{a}}}(g_{p})\mathrm{d}\tau
\end{equation*}
\begin{eqnarray*}
&=&\dfrac{1}{p^{2}}\int_{\mathfrak{m}_{p}}|x|_{p}^{s_{\alpha(\beta)}-\kappa_{\alpha(\beta)}}\mathsf{1}_{\mathfrak{m}_{p}^{m_{\alpha(\beta)}}}(x_{p})\overline{\eta_{\mathbf{a}}}\left(\dfrac{1}{x_{p}^{d}}\right) \mathrm{d}x_{p}\\
&=&\dfrac{1}{p^{2}}\sum_{\ell=m_{\alpha(\beta)}}^{+\infty}p^{-\ell(s_{\alpha(\beta)}-\kappa_{\alpha(\beta)}+1)}\int_{\mathbb{Z}_{p}^{\times}}\overline{\eta_{\mathbf{a}}}(p^{-d\ell}x_{p}^{d})\mathrm{d}x_{p}\\
&=&\begin{cases}
-\dfrac{1}{p^{3}}p^{-(s_{\alpha(\beta)}-\kappa_{\alpha(\beta)}+1)},& \mathrm{if}\; d=m_{\alpha(\beta)}=1
\\ 0, & \text{otherwise}
\end{cases} \\
&=& O(p^{-(3+\delta_{2})})
\end{eqnarray*}
for some $ \delta_{2}>0 $.\par 
If $ y \in E(f_{\mathbf{a}})(\mathbb{F}_{p}) $, we have for sufficiently small $ \delta>0 $, 
\begin{equation*}
\left| \int_{\rho_{p}^{-1}(y)}\mathsf{H}_{p}(\mathbf{s}-\mathbf{\kappa},g_{p})^{-1}\delta_{\varepsilon,p}(g_{p})\overline{\eta_{\mathbf{a}}}(g_{p})\mathrm{d}\tau \right|
\end{equation*}
\begin{eqnarray*}
&\leq & \int_{\rho_{p}^{-1}(y)}\mathsf{H}_{p}(\Re(\mathbf{s})-\mathbf{\kappa},g_{p})^{-1}\delta_{\varepsilon,p}(g_{p})\mathrm{d}\tau\\
&=& O(p^{-(2+\delta_{3})})
\end{eqnarray*}
for some $ \delta_{3}>0 $.\par
By the Lang-Weil estimates \cite{Lawe}:
\begin{equation*}
\#\left(\mathcal{D}_{p,B}^{\circ}\backslash E(f_{\mathbf{a}})\right)(\mathbb{F}_{p})=O(p^{2}),\quad \# \left(\mathcal{D}_{p,B}^{\circ}\cap E(f_{\mathbf{a}})\right)(\mathbb{F}_{p})=O(p),
\end{equation*}
we have for sufficiently small $ \delta>0 $,
\begin{equation*}
\sum_{y\in \mathcal{D}_{p,B}^{\circ}(\mathbb{F}_{p})}\int_{\rho_{p}^{-1}(y)}\mathsf{H}_{p}(\mathbf{s}-\mathbf{\kappa},g_{p})^{-1}\delta_{\varepsilon,p}(g_{p})\overline{\eta_{\mathbf{a}}}(g_{p})\mathrm{d}\tau =O(p^{-(1+\delta_{4})})
\end{equation*}
for some $ \delta_{4}>0 $.
\item If $ \#B\geq 2 $, then as in the proof of Proposition \ref{propG}, we have for $ \delta>0 $ sufficiently small,
\begin{equation*}
\sum_{y\in \mathcal{D}_{p,B}^{\circ}(\mathbb{F}_{p})}\int_{\rho_{p}^{-1}(y)}\mathsf{H}_{p}(\mathbf{s}-\mathbf{\kappa},g_{p})^{-1}\delta_{\varepsilon,p}(g_{p})\overline{\eta_{\mathbf{a}}}(g_{p})\mathrm{d}\tau =O(p^{-(1+\delta_{5})})
\end{equation*}
for some $ \delta_{5}>0 $.
\end{enumerate} 
The proposition then follows from the estimates above.
\end{proof}
Next we assume $ j_{p}(\mathbf{a})\neq 0 $.
\begin{prop}\label{propI}
Let $ p $ be a nonarchimedean place such that $ p\not \in S $ and $ j_{p}(\mathbf{a})\neq 0 $. Then there is a $ \delta>0 $, independent of $ \mathbf{a} $, such that the following function
\begin{equation*}
\mathbf{s}\mapsto \prod_{\alpha\in \mathcal{A}^{0}(\mathbf{a})}\prod_{\beta\in \mathcal{A}_{p}(\alpha)}\zeta_{\mathbb{Q}_{p,\beta}}(m_{\alpha}(s_{\alpha}-\kappa_{\alpha}+1))^{-1}\mathsf{Z}_{1,\varepsilon,p}(\mathbf{s},\eta_{\mathbf{a}})
\end{equation*}
is holomorphic on $ \mathsf{T}_{>-\delta} $ and moreover there is a $\delta'>0$, independent of $ \mathbf{a} $ such that we  have
\begin{equation*}
\left| \prod_{\alpha\in \mathcal{A}^{0}(\mathbf{a})}\prod_{\beta\in \mathcal{A}_{p}(\alpha)}\zeta_{\mathbb{Q}_{p,\beta}}(m_{\alpha}(s_{\alpha}-\kappa_{\alpha}+1))^{-1}\mathsf{Z}_{1,\varepsilon,p}(\mathbf{s},\eta_{\mathbf{a}})\right| \ll \left(1+H_{p}(\mathbf{a})^{-1}\right)^{\delta'}.
\end{equation*}
\end{prop}
\begin{proof}
The proof is analogous to that of \cite[Proposition 8.2]{PSTVA19}. As before by stratification we have
\begin{equation*}
\mathsf{Z}_{1,\varepsilon,p}(\mathbf{s},\eta_{\mathbf{a}})=\sum_{B\subset \mathcal{A}_{p}}\sum_{y\in \mathcal{D}_{p,B}^{\circ}(\mathbb{F}_{p})}\int_{\rho_{p}^{-1}(y)}\mathsf{H}_{p}(\mathbf{s}-\mathbf{\kappa},g_{p})^{-1}\delta_{\varepsilon,p}(g_{p})\overline{\eta_{\mathbf{a}}}(g_{p})\mathrm{d}\tau .
\end{equation*}
\begin{enumerate}
\item If $ B=\varnothing $, then the inner sum  equals to a constant.
\item If $ B=\{\beta\} $, then $ \beta $ lies above a unique $ \alpha(\beta)\in \mathcal{A} $. There are two cases, $ \alpha(\beta)\in \mathcal{A}^{0}(\mathbf{a}) $ or $ \alpha(\beta)\in \mathcal{A}^{\geq 1}(\mathbf{a}) $.\par 
If $ \alpha(\beta)\in \mathcal{A}^{0}(\mathbf{a}) $, then as in the proof of Proposition \ref{propH}, for sufficiently small $ \delta>0 $, the inner sum is 
\begin{equation*}
p^{-m_{\alpha(\beta)}(s_{\alpha(\beta)}-\kappa_{\alpha(\beta)}+1)}\left(c+O(p^{-\delta_{1}})\right)
\end{equation*}
for some constants $ c $ and $ \delta_{1}>0 $.\par
If $ \alpha(\beta)\in \mathcal{A}^{\geq 1}(\mathbf{a}) $, we set $ d:=d_{\alpha(\beta)}(f_{\mathbf{a}}) $, there are two cases. If $ y\not \in E(f_{\mathbf{a}})(\mathbb{F}_{p}) $, then by \cite[Lemma 5.4]{PSTVA19}, we have for sufficiently small $ \delta>0 $,
\begin{equation*}
\int_{\rho_{p}^{-1}(y)}\mathsf{H}_{p}(\mathbf{s}-\mathbf{\kappa},g_{p})^{-1}\delta_{\varepsilon,p}(g_{p})\overline{\eta_{\mathbf{a}}}(g_{p})\mathrm{d}\tau
\end{equation*}
\begin{eqnarray*}
&=&\dfrac{1}{p^{2}}\int_{\mathfrak{m}_{p}}|x|_{p}^{s_{\alpha(\beta)}-\kappa_{\alpha(\beta)}}\mathsf{1}_{\mathfrak{m}_{p}^{m_{\alpha(\beta)}}}(x_{p})\overline{\eta_{\mathbf{a}}}\left(\dfrac{p^{j_{p}(\mathbf{a})}}{x_{p}^{d}}\right) \mathrm{d}x_{p}\\
&=&\dfrac{1}{p^{2}}\sum_{\ell=m_{\alpha(\beta)}}^{+\infty}p^{-\ell(s_{\alpha(\beta)}-\kappa_{\alpha(\beta)}+1)}\int_{\mathbb{Z}_{p}^{\times}}\overline{\eta_{\mathbf{a}}}(p^{-d\ell +j_{p}(\mathbf{a})}x_{p}^{d})\mathrm{d}x_{p}\\
&=& O\left(\dfrac{|j_{p}(\mathbf{a})|}{p^{2}}\right).
\end{eqnarray*}
The implied constant can be taken independent of $ \mathbf{a} $ since there are only finitely many choices for $ d_{\alpha}(f_{\mathbf{a}}) $ by \cite[Before Lemma 3.4.1]{CT-additive}, which says that the functions $ \mathbf{a}\mapsto d_{\alpha}(f_{\mathbf{a}}) $ on $ G\backslash \{0\} $ are upper semicontinuous and thus they do not decrease under specialization.\par 
If $ y \in E(f_{\mathbf{a}})(\mathbb{F}_{p}) $, for $ \delta>0 $ sufficiently small we have
\begin{equation*}
\left| \int_{\rho_{p}^{-1}(y)}\mathsf{H}_{p}(\mathbf{s}-\mathbf{\kappa},g_{p})^{-1}\delta_{\varepsilon,p}(g_{p})\overline{\eta_{\mathbf{a}}}(g_{p})\mathrm{d}\tau \right|
\end{equation*}
\begin{eqnarray*}
&\leq & \int_{\rho_{p}^{-1}(y)}\mathsf{H}_{p}(\Re(\mathbf{s})-\mathbf{\kappa},g_{p})^{-1}\delta_{\varepsilon,p}(g_{p})\mathrm{d}\tau\\
&=& O(p^{-(2+\delta')})
\end{eqnarray*}
for some $ \delta'>0 $. Thus using the Lang-Weil estimates \cite{Lawe}:
\begin{equation*}
\#\left(\mathcal{D}_{p,B}^{\circ}\backslash E(f_{\mathbf{a}})\right)(\mathbb{F}_{p})=O(p^{2}),\quad \# \left(\mathcal{D}_{p,B}^{\circ}\cap E(f_{\mathbf{a}})\right)(\mathbb{F}_{p})=O(p),
\end{equation*}
we have
\begin{equation*}
\sum_{y\in \mathcal{D}_{p,B}^{\circ}(\mathbb{F}_{p})}\int_{\rho_{p}^{-1}(y)}\mathsf{H}_{p}(\mathbf{s}-\mathbf{\kappa},g_{p})^{-1}\delta_{\varepsilon,p}(g_{p})\overline{\eta_{\mathbf{a}}}(g_{p})\mathrm{d}\tau =O(|j_{p}(\mathbf{a})|).
\end{equation*}
\item If $ \#B\geq 2 $, then as in the proof of Proposition \ref{propG}, we have
\begin{equation*}
\sum_{y\in \mathcal{D}_{p,B}^{\circ}(\mathbb{F}_{p})}\int_{\rho_{p}^{-1}(y)}\mathsf{H}_{p}(\mathbf{s}-\mathbf{\kappa},g_{p})^{-1}\delta_{\varepsilon,p}(g_{p})\overline{\eta_{\mathbf{a}}}(g_{p})\mathrm{d}\tau =O(p^{-(1+\delta'')})
\end{equation*}
for some $ \delta''>0 $.
\end{enumerate}
The proposition then follows from the estimates above.
\end{proof}

\subsubsection{Places of bad reduction}
Again let $ p\not \in S $. We assume the integral model $ \mathcal{X} $ has bad reduction at $ p $ or the metrics at $ p $ are not induced by our integral model.
\begin{prop}\label{propJJ}
The height function $ \mathsf{Z}_{1,\varepsilon,p}(\mathbf{s},\eta_{\mathbf{a}}) $ is holomorphic in $ \mathbf{s} $ when $ \Re(s_{\alpha})>\kappa_{\alpha}-1 $ for $ \alpha\in \mathcal{A}^{0}(\mathbf{a}) $ with $ \varepsilon_{\alpha}<1 $. For  $ \delta>0 $ there are $ \delta'>0 $ and $ C_{\delta}>0 $ such that
\begin{equation}\label{z1}
\left| \mathsf{Z}_{1,\varepsilon,p}(\mathbf{s},\eta_{\mathbf{a}})\right| < C_{\delta}\left(1+H_{\infty}(\mathbf{a})\right)^{\delta'}
\end{equation}
when $ \Re(s_{\alpha})>\kappa_{\alpha}-1+\delta $ for $ \alpha\in \mathcal{A}^{0}(\mathbf{a}) $ with $\varepsilon_{\alpha}<1 $.
\end{prop}
\begin{proof}
The first assertion follows analogous arguments as in the proof of \cite[Lemma 4.1]{volume} or \cite[Lemma 8.2]{VecIII}. As for the estimate (\ref{z1}) we can take it from \cite[Proposition 8.3]{PSTVA19} since the proof of \cite[Proposition 8.3]{PSTVA19} concerns the affine spaces and the  group action is not involved.
\end{proof}

\subsection{Places in $S$}
For the case $ v\in S $, $ \delta_{\varepsilon,v}\equiv 1 $ by definition. 
\begin{prop}\label{propJ}
Let $ v\in S $. Then we have
\begin{enumerate}
\item \label{II-I} The function
\begin{equation*}
\mathbf{s}\mapsto \mathsf{Z}_{1,\varepsilon,v}(\mathbf{s},\eta_{\mathbf{a}})
\end{equation*}
is holomorphic in the domain $ \Re(s_{\alpha})>\kappa_{\alpha}-1+\delta $ for all $ \alpha\in \mathcal{A} $ and any $ \delta>0 $. For any $ n>0 $, there exists $ m_{n}>0 $ such that 
\begin{equation}\label{5.4.1pro}
\left| \prod_{v\in S}\mathsf{Z}_{1,\varepsilon,v}(\mathbf{s},\eta_{\mathbf{a}})\right| \ll_{n}\dfrac{(1+|\mathbf{s}|)^{m_{n}}}{(1+H_{\infty}(\mathbf{a}))^{n}}
\end{equation}
in the above domain.
\item If $ L=\sum_{\alpha\in \mathcal{A}}\lambda_{\alpha}D_{\alpha} $ is a big divisor on $ X $, let
\begin{equation*}
a:=\tilde{a}((X,D_{\mathrm{red}}),L),\quad b:=b(\mathbb{Q}_{v},(X,D_{\mathrm{red}}),L,f_{\mathbf{a}})
\end{equation*}
be as defined in \cite[\S 4]{PSTVA19}. Then there is a constant $ \delta>0 $ such that the function
\begin{equation*}
s\mapsto \zeta_{\mathbb{Q}_{v}}(s-a)^{-b}\mathsf{Z}_{1,\varepsilon,v}(sL,\eta_{\mathbf{a}})
\end{equation*}
admits a holomorphic continuation to the domain  $ \Re(s)>a-\delta $. Moreover for any $ n>0 $ there exists $ m_{n}>0 $ such that 
\begin{equation}\label{5.4.2pro}
\left| \prod_{v\in S}\zeta_{\mathbb{Q}_{v}}(s-a)^{-b}\mathsf{Z}_{1,\varepsilon,v}(sL,\eta_{\mathbf{a}})\right| \ll_{n}\dfrac{(1+|s|)^{m_{n}}}{(1+H_{\infty}(\mathbf{a}))^{n}}
\end{equation}
in the above domain.
\end{enumerate}
\end{prop}
\begin{proof}
\begin{enumerate}
\item Since $ \delta_{\varepsilon,v}\equiv 1 $,
\begin{equation*}
\mathsf{Z}_{1,\varepsilon,v}(\mathbf{s},\eta_{\mathbf{a}})=\int_{G(\mathbb{Q}_{v})}\mathsf{H}_{v}(\mathbf{s},g_{v})^{-1}\overline{\eta_{\mathbf{a}}}(g_{v})\mathrm{d}g_{v},
\end{equation*}
then (\ref{II-I}) follows from \cite[Proposition 8.1]{VecIII}. Note that although \cite[Proposition 8.1]{VecIII} is stated for vector groups, it is also applicable to the Heisenberg group since the proof of it concerns the affine spaces and the group action is not involved.
\item The proof is analogous to that of \cite[Proposition 3.4.4]{CT-additive} and we shall see as \cite[Proposition 3.4.4]{CT-additive}, the holomorphic continuation is insured by the positivity $ d_{\alpha}(f_{\mathbf{a}})>0 $ for some $ \alpha\in \mathcal{A} $.\par 
There is a stratification of $ X $ by locally closed subvarieties
\begin{equation*}
D_{A}:=\bigcap_{\alpha\in A}D_{\alpha},\qquad D_{A}^{\circ}:=D_{A}\backslash \bigcup_{A\subsetneq A'}D_{A'}
\end{equation*}
where $ A\subset \mathcal{A} $.\par
Recall that the divisor of $ f_{\mathbf{a}} $ is
\begin{equation*}
\mathrm{div}(f_{\mathbf{a}})=E(f_{\mathbf{a}})-\sum_{\alpha\in \mathcal{A}}d_{\alpha}(f_{\mathbf{a}})D_{\alpha},
\end{equation*}
with $ d_{\alpha}(f_{\mathbf{a}})\geq 0 $ and there is at least one $ \alpha $ such that $ d_{\alpha}(f_{\mathbf{a}})> 0 $. The function $ f_{\mathbf{a}} $ may have indeterminacies, so we apply a resolution of indeterminacies by \cite[Lemma 3.4.1]{CT-additive} and write $ \pi: Y\rightarrow P\times X $ for the resolution where $ P $ is a fixed strata of the decomposition $ \mathbb{P}^{2}=\sqcup_{i}P_{i} $ of $ \mathbb{P}^{2} $ into locally closed subsets. The coefficients $ d_{\alpha}(f_{\mathbf{a}}) $ are constant on $ P $ and we denote them by $ d_{\alpha} $. Now the divisor 
of $ f_{\mathbf{a}} $ can be written as
\begin{equation*}
\mathrm{div}(f_{\mathbf{a}})=H_{\mathbf{a}}-\sum_{\alpha\in A}d_{\alpha}D_{\alpha}'-\sum_{\beta\in B}e_{\beta}E_{\beta}
\end{equation*}
where the divisor $ H_{\mathbf{a}} $ is the strict transform of $ E(f_{\mathbf{a}}) $ in $ G $, $ D_{\alpha}' $ is the strict transform of $ D_{\alpha} $ and the divisors $ E_{\beta} $ are the exceptional divisors of the resolution $ \pi $.\par
Let us consider local coordinates $ \left((x_{\alpha})_{\alpha\in A},(y_{\beta})_{\beta\in B}\right) $ for a point $ \mathbf{x}\in Y  $, where $ x_{\alpha} $ is a local equation of $ D_{\alpha}' $ for each $ \alpha $ and $ y_{\beta} $ is a local equation of $ E_{\beta} $ for each $ \beta $. We extend $ \left((x_{\alpha})_{\alpha\in A},(y_{\beta})_{\beta\in B}\right) $ to a system of local coordinates $ \left((x_{\alpha})_{\alpha\in A},(y_{\beta})_{\beta\in B},(z_{\gamma})_{\gamma\in C}\right) $ in a neighborhood $ \mathcal{U}=P\times \mathcal{U}_{0} $ of $ \mathbf{x} $ in $ Y(\mathbb{Q}_{v}) $, here $ \mathcal{U}_{0} $ is a subset of $ X(\mathbb{Q}_{v}) $. From above construction, the function
\begin{equation*}
u_{\mathbf{a}}:\,x\mapsto \prod_{\alpha\in A}x_{\alpha}^{d_{\alpha}}\prod_{\beta\in B}y_{\beta}^{e_{\beta}}f_{\mathbf{a}}
\end{equation*}
is regular on $ \mathcal{U} $ and homogeneous of degree one in $ \mathbf{a} $. Thus there is an upper bound
\begin{equation*}
\vert u_{\mathbf{a}}(\mathbf{x})\vert\ll\Vert\mathbf{a}\Vert
\end{equation*}
for $ \mathbf{x}\in\mathcal{U}_{0} $. Applying a partition of unity, the integral $ \mathsf{Z}_{1,\varepsilon,v}(\mathbf{s},\eta_{\mathbf{a}}) $ now takes the form of the summation of
\begin{equation}\label{3.4.4}
\int_{\mathcal{U}_{0}}\prod_{\alpha\in A}|x_{\alpha}|^{\lambda_{\alpha}s-\kappa_{\alpha}}\prod_{\beta\in B}|y_{\beta}|^{\lambda_{\beta}(s)}\overline{\psi_{1}}\left(u_{\mathbf{a}}\prod_{\alpha\in A}x_{\alpha}^{-d_{\alpha}}\prod_{\beta\in B}y_{\beta}^{-e_{\beta}} \right)\prod_{\alpha\in A}\mathrm{d}x_{\alpha}\prod_{\beta\in B}\mathrm{d}y_{\beta}\prod_{\gamma\in C}\mathrm{d}z_{\gamma},
\end{equation}
where $ \lambda_{\beta}(s)\geq 0 $ for each $ \beta\in B $ (See \cite[\S 3.4]{CT-additive}).\par 
Note that $ u_{\mathbf{a}} $ does not vanish on $ \mathcal{U} $ because of the resolution of indeterminacies of $ f_{\mathbf{a}} $. We now show the holomorphic continuation. We first integrate with respect to a variable $ x_{\alpha_{0}} $ such that $ d_{\alpha_{0}}>0 $. Applying \cite[Proposition 2.4.1]{CT-additive}, the integral (\ref{3.4.4}) then can be written as
\begin{equation*}
\int_{\mathcal{U}_{0}}\prod_{\alpha\in A\atop \alpha\neq\alpha_{0}}|x_{\alpha}|^{\lambda_{\alpha}s-\kappa_{\alpha}}\prod_{\beta\in B}|y_{\beta}|^{\lambda_{\beta}(s)}F(\mathbf{x}',\mathbf{a})\prod_{\alpha\in A\atop \alpha\neq\alpha_{0}}\mathrm{d}x_{\alpha}\prod_{\beta\in B}\mathrm{d}y_{\beta}\prod_{\gamma\in C}\mathrm{d}z_{\gamma}
\end{equation*}
where $ \mathbf{x}'=\left((x_{\alpha})_{\alpha\neq\alpha_{0}},(y_{\beta})_{\beta\in B},(z_{\gamma})_{\gamma\in C}\right) $ and
\begin{eqnarray*}
\vert F(\mathbf{x}',\mathbf{a})\vert &\ll & \left(\Vert \mathbf{a}\Vert\prod_{\alpha\in A\atop \alpha\neq\alpha_{0}}|x_{\alpha}|^{-d_{\alpha}}\prod_{\beta\in B}|y_{\beta}|^{-e_{\beta}}\right)^{-1/d_{\alpha_{0}}}\\
& \ll &\Vert \mathbf{a}\Vert^{-1/d_{\alpha_{0}}}\prod_{\alpha\in A\atop \alpha\neq\alpha_{0}}|x_{\alpha}|^{d_{\alpha}/d_{\alpha_{0}}}\prod_{\beta\in B}|y_{\beta}|^{e_{\beta}/d_{\alpha_{0}}}.
\end{eqnarray*}
Since there is at least one $ \alpha $ such that $ d_{\alpha}> 0 $, the positive exponent $ d_{\alpha}/d_{\alpha_{0}} $ of $ |x_{\alpha}| $ insures absolute convergence of the integral (\ref{3.4.4}) in a neighborhood of $ a $, the assertion on holomorphic continuation is proved. We adopt the estimate (\ref{5.4.2pro}) from \cite[Proposition 3.4.4, Lemma 3.5.2]{CT-additive}.
\end{enumerate} 
\end{proof}

\subsection{Euler products}
Finally we consider the Euler product from Lemma \ref{lemma decomposition}:
\begin{equation*}
\mathsf{Z}_{1,\varepsilon}(\mathbf{s},\eta_{\mathbf{a}})=\prod_{v\in \Val(\mathbb{Q})}\mathsf{Z}_{1,\varepsilon,v}(\mathbf{s},\eta_{\mathbf{a}}),
\end{equation*}
where $ \mathbf{a}\in \Lambda_{1} $. For every $ \alpha\in \mathcal{A} $ we set
\begin{equation*}
\zeta_{F_{\alpha},S}(s)=\prod_{v\not \in S}\prod_{\beta\in \mathcal{A}_{v}(\alpha)}\zeta_{\mathbb{Q}_{v},\beta}(s).
\end{equation*}
\begin{prop}\label{propP}
Assume that $ (X,D_{\varepsilon}) $ is klt. Then there is a constant $ \delta>0 $, independent of $ \mathbf{a} $, such that the function
\begin{equation*}
\mathbf{s}\mapsto \left(\prod_{\alpha\in \mathcal{A}^{0}(\mathbf{a})}\zeta_{F_{\alpha},S}(m_{\alpha}(s_{\alpha}-\kappa_{\alpha}+1))\right)^{-1} \mathsf{Z}_{1,\varepsilon}(\mathbf{s},\eta_{\mathbf{a}})
\end{equation*}
is holomorphic on $ \mathsf{T}_{>-\delta} $. \par 
Moreover, for any $ n>0 $ there exists $ m_{n}>0 $ such that  
\begin{equation*}\label{estimate3}
\left| \left(\prod_{\alpha\in \mathcal{A}^{0}(\mathbf{a})}\zeta_{F_{\alpha},S}(m_{\alpha}(s_{\alpha}-\kappa_{\alpha}+1))\right)^{-1} \mathsf{Z}_{1,\varepsilon}(\mathbf{s},\eta_{\mathbf{a}})\right| \ll_{n} \dfrac{(1+|\mathbf{s}|)^{m_{n}}}{(1+H_{\infty}(\mathbf{a}))^{n}}.
\end{equation*}
\end{prop}
\begin{proof}
This follows from Propositions \ref{propH}, \ref{propI}, \ref{propJJ} and \ref{propJ}, together with the estimate (\ref{eqD}).
\end{proof}

\section{Height integrals III}\label{ass}
In this section, we study the height integral $ \mathsf{Z}_{2,\varepsilon}(\mathbf{s}) $ from Lemma \ref{lemma decomposition}.\par
Let $ Y\subset X $ be the induced  compactification of $ \mathrm{U}\subset G $, i.e., the Zariski closure of $ \mathrm{U} $ in $ X $, and we denote by $\mathcal{Y}\, (\mathrm{resp}.\: \mathcal{D}^{\mathcal{Y}}) $ the Zariski closure of $Y\,(\mathrm{resp}.\: D^{Y}) $ in $ \mathcal{X} $ where $D^{Y}:= Y\backslash \mathrm{U} $.
\begin{assu}[{\cite[Assumption 4.7]{ST04}}] \label{assu}
We assume that the boundary $D^{Y} $ is a strict normal crossing divisor whose components are intersections of the boundary components of $  X$ with $ Y $:
\begin{equation*}
Y\backslash \mathrm{U}=\bigcup_{\alpha\in \mathcal{A}^{Y}}D_{\alpha}^{Y}
\end{equation*}
where $ \mathcal{A}^{Y}\subseteq \mathcal{A} $  and $ D_{\alpha}^{Y}= D_{\alpha}\cap Y $ for all $ \alpha\in \mathcal{A}^{Y} $. 
\end{assu} 
We have for the anticanonical divisor of $ Y $
\begin{equation*}
-K_{Y}=\sum_{\alpha\in \mathcal{A}^{Y}}\kappa_{\alpha}^{Y}D_{\alpha}^{Y},
\end{equation*}
with
\begin{equation}\label{4}
\kappa_{\alpha}^{Y}\leq \kappa_{\alpha}
\end{equation}
for all $ \alpha\in \mathcal{A}^{Y}  $ \cite[p.22]{ST04}. For nonzero $ a\in \mathbb{Q} $ denote by $ f_{a} $ the linear form on $ G $
\begin{equation*}
z\mapsto a\cdot z.
\end{equation*}
The linear form $ f_{a} $ defines an automorphic character $ \psi_{a} $ of $ \mathrm{U}(\mathbb{A}_{\mathbb{Q}})/ \mathrm{U}(\mathbb{Q}) $ by
\begin{equation*}
\psi_{a}\left(g(0,z,y)\right)=\psi_{1}(az)
\end{equation*}
where $ \psi_{1} $ is the Tate character \cite{Tate}. Write the divisor of $ f_{a} $ as
\begin{equation*}
\mathrm{div}(f_{a})=E(f_{a})-\sum_{\alpha\in \mathcal{A}^{Y}}d_{\alpha}(f_{a})D_{\alpha}^{Y}
\end{equation*}
where $ E(f_{a}) $ is the hyperplane along which $ f_{a} $ vanishing in $ Y $. Define
\begin{equation*}
\mathcal{A}^{0}:=\{\alpha\in \mathcal{A}^{Y}|d_{\alpha}(f_{a})=0\},
\end{equation*}
\begin{equation*}
\mathcal{A}^{\geq 1}:=\{\alpha\in \mathcal{A}^{Y}|d_{\alpha}(f_{a})\geq 1\},
\end{equation*}
note that the set $ \mathcal{A}^{0} $ is a proper subset of $ \mathcal{A} $ by \cite[p.22]{ST04}. For $\psi_{a} $ with $a\in \frac{1}{n(\mathsf{K})}\mathbb{Z}$ (see (\ref{nk}) for $ n(\mathsf{K}) $), we define $ S_{\psi_{a}} $ the set of places by 
\begin{equation}\label{definition of psi}
S_{\psi_{a}}:=\{\text{prime}\: p:  p| n(\mathsf{K})a\}\cup S.
\end{equation}
For every place $ v\in \Val(\mathbb{Q}) $ we denote by $ |a|_{v} $ the $ v $-adic absolute value of $ a $ and denote
\begin{equation*}
|a|_{\mathrm{fin}}=\prod_{p\in\Val(\mathbb{Q})_{\mathrm{fin}}}|a|_{p},
\end{equation*}
then by the product formula we have
\begin{equation}\label{eqE}
|a|_{\infty}=|a|_{\mathrm{fin}}^{-1}.
\end{equation}
We are going to analyse the integral from Lemma \ref{lemma decomposition}
\begin{equation}
\int_{G(\mathbb{A}_{\mathbb{Q}})}\mathsf{H}(\mathbf{s},g)^{-1}\overline{\omega}^{\psi_{a}}(g)\delta_{\varepsilon}(g)\mathrm{d}g.
\end{equation}
Fix a nonzero $ a\in \mathbb{Q} $, we write the character $ \psi_{a}=\prod\limits_{v\in \Val(\mathbb{Q})}\psi_{a,v} $. For a place $ p\not \in S_{\psi_{a}} $, recall from \cite[Definition 3.14]{ST04} that the normalized spherical function $ f_{p} $ on $ G(\mathbb{Q}_{p}) $ is defined to be
\begin{equation*}
f_{p}(g_{p})=\langle \pi_{\psi_{a,p}}(g_{p})e_{p}, e_{p}\rangle,
\end{equation*}
where $ \langle \cdot,\cdot \rangle $ is the standard inner product on $ \mathsf{L}^{2}(\mathbb{Q}_{p}) $, $ \pi_{\psi_{a,p}} $ is the representation of $ G(\mathbb{Q}_{p}) $ induced by a nontrivial character $ \psi_{a,p} $ of $ \mathbb{Q}_{p} $ and $ e_{p}\in V_{p}^{\mathsf{K}_{p}} $ is a normalized vector of the $ \mathsf{K}_{p} $-invariant subspace of a representation space $ V_{p} $ for $ \pi_{\psi_{a,p}} $, for $ \mathsf{K}_{p} $ see Lemma \ref{lemma good reduction}.\par
We denote
\begin{equation*}
\mathbb{A}^{S_{\psi_{a}}}={\prod}'_{v\not\in S_{\psi_{a}}}\mathbb{Q}_{v},\qquad \mathbb{A}_{S_{\psi_{a}}}=\prod_{v\in S_{\psi_{a}}}\mathbb{Q}_{v},
\end{equation*}
where $ \prod' $ means the restricted product. \par
There are decompositions
\begin{equation*}
G(\mathbb{A})=G(\mathbb{A}^{S_{\psi_{a}}})\cdot G(\mathbb{A}_{S_{\psi_{a}}})={\prod}'_{v\not \in S_{\psi_{a}}}G(\mathbb{Q}_{v})\cdot \prod_{v \in S_{\psi_{a}}}G(\mathbb{Q}_{v})
\end{equation*}
and for $ g\in G(\mathbb{A}) $,
\begin{equation*}
g=g^{S_{\psi_{a}}}\cdot g_{S_{\psi_{a}}}={\prod}_{v\not \in S_{\psi_{a}}}g_{v}\cdot \prod_{v\in S_{\psi_{a}}}g_{v}.
\end{equation*}
From \cite[Lemma 3.15, Corollary 3.16]{ST04} we have
\begin{equation}
\int_{G(\mathbb{A}_{\mathbb{Q}})}\mathsf{H}(\mathbf{s},g)^{-1}\overline{\omega}^{\psi_{a}}(g)\delta_{\varepsilon}(g)\mathrm{d}g
\end{equation}
\begin{equation*}
=\left(\prod_{p\not \in S_{\psi_{a}}}\int_{G(\mathbb{Q}_{p})}\mathsf{H}_{p}(\mathbf{s},g_{p})^{-1}f_{p}(g_{p})\delta_{\varepsilon,p}(g_{p})\mathrm{d}g_{p}\right) \int_{G(\mathbb{A}_{S_{\psi_{a}}})}\mathsf{H}(\mathbf{s},g_{S_{\psi_{a}}})^{-1}\overline{\omega_{S_{\psi_{a}}}}^{\psi_{a}}(g_{S_{\psi_{a}}})\delta_{\varepsilon}(g_{S_{\psi_{a}}})\mathrm{d}g_{S_{\psi_{a}}},
\end{equation*}
where $ \omega_{S_{\psi_{a}}} $ is the restriction of $ \omega $ to $ G(\mathbb{A}_{S_{\psi_{a}}}) $.\par 
As in the cases for $ \mathsf{Z}_{0,\varepsilon}(\mathbf{s},g) $ and $ \mathsf{Z}_{1,\varepsilon}(\mathbf{s},g) $, we write
\begin{equation*}
D^{Y}\otimes_{\mathbb{Q}}\mathbb{Q}_{v}=\bigcup_{\beta\in \mathcal{A}_{v}^{Y}}D_{v,\beta}^{Y},
\end{equation*}
\begin{equation*}
D_{\alpha}^{Y}\otimes_{\mathbb{Q}}\mathbb{Q}_{v}=\bigcup_{\beta\in \mathcal{A}_{v}^{Y}(\alpha)}D_{v,\beta}^{Y}
\end{equation*}
for the decompositions into irreducible components.\par
Let $ \beta\in \mathcal{A}_{v}^{Y} $, we denote the field of definition for one of the geometric irreducible components of $ D_{v,\beta}^{Y} $ by $ \mathbb{Q}_{v,\beta}^{Y} $ and for any subset $ B\subseteq \mathcal{A}_{v}^{Y} $, we define
\begin{equation*}
D_{v,B}^{Y}:=\bigcap_{\beta\in B}D_{v,\beta}^{Y},\qquad D_{v,B}^{\circ,Y}:=D_{v,B}^{Y}\backslash \bigcup_{B\subsetneq B'\subset \mathcal{A}_{v}^{Y}}\left(\bigcap_{\beta\in B'}D_{v,\beta}^{Y}\right)
\end{equation*}
where we set that $ D_{v,\varnothing}^{Y}=X_{\mathbb{Q}_{v}} $ and $ D_{v,\varnothing}^{\circ,Y}=G_{\mathbb{Q}_{v}} $. For $ v\not \in S $, we denote by $ \mathcal{D}_{v,B}^{Y} $ the Zariski closure of $D_{v,B}^{Y} $ in $ \mathcal{Y}\otimes_{\mathbb{Z}_{S}}\mathbb{Z}_{v} $. We define $ \mathcal{D}_{v,B}^{\circ,Y} $ analogously.

\subsection{Places not in $ S_{\psi_{a}} $}
Let $ p\not \in S_{\psi_{a}} $. By \cite[Lemma 3.17]{ST04}, we have
\begin{equation*}
\mathsf{Z}_{2,\varepsilon,p}(\mathbf{s},\omega^{\psi_{a}}):=\int_{G(\mathbb{Q}_{p})}\mathsf{H}_{p}(\mathbf{s},g_{p})^{-1}f_{p}(g_{p})\delta_{\varepsilon,p}(g_{p})\mathrm{d}g_{p}=\int_{\mathrm{U}(\mathbb{Q}_{p})}\mathsf{H}_{p}(\mathbf{s},u_{p})^{-1}\psi_{a,p}(u_{p})\delta_{\varepsilon,p}(u_{p})\mathrm{d}u_{p}.
\end{equation*}
\subsubsection{Places of good reduction}
We assume the integral model $ \mathcal{Y} $ has good reduction at $ p $ and the metrics at $ p $ are induced by our integral model. For a nonzero $ a\in \mathbb{Q} $, we denote by $ v_{p}(a) $ the $ p $-adic valuation of $ a $. There are two cases, $ v_{p}(a)=0 $ or $ v_{p}(a)\neq 0 $. We first assume $ v_{p}(a)=0 $.
\begin{prop}\label{propK}
We have
\begin{enumerate}
\item There exists $\delta>0$, independent of $ a $ such that the function
\begin{equation*}
\mathbf{s}\mapsto \prod_{\alpha\in \mathcal{A}^{0}}\prod_{\beta\in \mathcal{A}_{p}^{Y}(\alpha)}\zeta_{\mathbb{Q}_{p,\beta}^{Y}}(m_{\alpha}(s_{\alpha}-\kappa_{\alpha}^{Y}+1))^{-1}\mathsf{Z}_{2,\varepsilon,p}(\mathbf{s},\omega^{\psi_{a}})
\end{equation*}
is holomorphic on $ \mathsf{T}_{>-\delta} $. If $ \varepsilon_{\alpha}=1 $, we take $ \zeta_{\mathbb{Q}_{p,\beta}^{Y}}(m_{\alpha}(s_{\alpha}-\kappa_{\alpha}^{Y}+1))^{-1}=1 $.
\item There exists $\delta'>0$, independent of $ a $ such that
\begin{equation*}
\prod_{\alpha\in \mathcal{A}^{0}}\prod_{\beta\in \mathcal{A}_{p}^{Y}(\alpha)}\zeta_{\mathbb{Q}_{p,\beta}^{Y}}(m_{\alpha}(s_{\alpha}-\kappa_{\alpha}^{Y}+1))^{-1}\mathsf{Z}_{2,\varepsilon,p}(\mathbf{s},\omega^{\psi_{a}})=1+O(p^{-(1+\delta')}),
\end{equation*}
for $\mathbf{s}\in \mathsf{T}_{>-\delta} $.
\end{enumerate}
\end{prop}
\begin{proof}
The proof is analogous to that of Proposition \ref{propH}.
\end{proof}

Next we assume $ v_{p}(a)\neq 0 $.
\begin{prop}\label{propL}
Let $ p $ be a nonarchimedean place such that $ p\not \in S_{\psi_{a}} $ and $ j_{p}(a)\neq 0 $. Then there is a constant $ \delta>0 $, independent of $ a $, such that the following function
\begin{equation*}
\mathbf{s}\mapsto \prod_{\alpha\in \mathcal{A}^{0}}\prod_{\beta\in \mathcal{A}_{p}^{Y}(\alpha)}\zeta_{\mathbb{Q}_{p,\beta}^{Y}}(m_{\alpha}(s_{\alpha}-\kappa_{\alpha}^{Y}+1))^{-1}\mathsf{Z}_{2,\varepsilon,p}(\mathbf{s},\omega^{\psi_{a}})
\end{equation*}
is holomorphic on $ \mathsf{T}_{>-\delta} $ and moreover there exists $\delta'>0$, independent of $ a $ such that
\begin{equation*}
\left| \prod_{\alpha\in \mathcal{A}^{0}}\prod_{\beta\in \mathcal{A}_{p}^{Y}(\alpha)}\zeta_{\mathbb{Q}_{p,\beta}^{Y}}(m_{\alpha}(s_{\alpha}-\kappa_{\alpha}^{Y}+1))^{-1}\mathsf{Z}_{2,\varepsilon,p}(\mathbf{s},\omega^{\psi_{a}})\right| \ll \left(1+\left|a\right|_{p}^{-1}\right)^{\delta'}.
\end{equation*}
\end{prop}
\begin{proof}
The proof is analogous to that of Proposition \ref{propI}.
\end{proof}

\subsubsection{Places of bad reduction}
Again let $ p\not \in S_{\psi_{a}} $. We assume the integral model $ \mathcal{Y} $ has bad reduction at $ p $ or the metrics at $ p $ are not induced by our integral model.
\begin{prop}\label{propLL}
The function $ \mathsf{Z}_{2,\varepsilon,p}(\mathbf{s},\omega^{\psi_{a}}) $ is holomorphic in $ \mathbf{s} $ when $ \Re(s_{\alpha})>\kappa_{\alpha}-1 $ for $ \alpha\in \mathcal{A}^{0} $ with $ \varepsilon_{\alpha}<1 $. For $ \delta>0 $ there are constants $ \delta'>0 $ and $ C_{\delta}>0 $ such that
\begin{equation*}
\vert \mathsf{Z}_{2,\varepsilon,p}(\mathbf{s},\omega^{\psi_{a}})\vert < C_{\delta}\left(1+\left|a\right|_{\infty}\right)^{\delta'}
\end{equation*}
when $ \Re(s_{\alpha})>\kappa_{\alpha}-1+\delta $ for $ \alpha\in \mathcal{A}^{0} $ with $\varepsilon_{\alpha}<1 $.
\end{prop}
\begin{proof}
The proof is analogous to that of Proposition \ref{propJJ}.
\end{proof}

\subsection{Places in $S_{\psi_{a}}$}
Let $ v\in S_{\psi_{a}} $. We first compute $ \omega_{S_{\psi_{a}}}(g_{S_{\psi_{a}}}) $. Let $ \mathcal{S}(\mathbb{A}_{\mathbb{Q}})\subset \mathsf{L}^{2}(\mathbb{A}_{\mathbb{Q}}) $ denote the space of Schwartz-Bruhat functions. Recall from \cite[\S 3]{ST04} that for a function $ \varphi\in \mathcal{S}(\mathbb{A}_{\mathbb{Q}}) $, the theta distribution is defined as
\begin{equation}\label{6.5}
\Theta(\varphi):=\sum_{\gamma\in \mathbb{Q}}\varphi(\gamma).
\end{equation}
The theta distribution gives a map
\begin{equation*}
j_{\psi_{a}}: \mathcal{S}(\mathbb{A}_{\mathbb{Q}})\rightarrow \mathsf{L}^{2}(G(\mathbb{Q})\backslash G(\mathbb{A}_{\mathbb{Q}}))
\end{equation*}
defined by
\begin{equation*}
j_{\psi_{a}}(\varphi)(g)=\Theta(\pi_{\psi_{a}}(g)\varphi)
\end{equation*}
where $ \pi_{\psi_{a}} $ is the representation of $ G(\mathbb{Q}) $ induced by $ \psi_{a} $  \cite[\S 3.5]{ST04}. Let $ \mathsf{K}^{S_{\psi_{a}}}:=\prod_{p\not\in S_{\psi_{a}}}\mathsf{K}_{p}=\prod_{p\not\in S_{\psi_{a}}}G(\mathbb{Z}_{p}) $, see Lemma \ref{lemma good reduction} for $ \mathsf{K}_{p} $. From the last formula of the proof of \cite[Lemma 3.15]{ST04} we have
\begin{equation}\label{6.6}
\omega_{S_{\psi_{a}}}(g_{S_{\psi_{a}}})=\int_{\mathsf{K}^{S_{\psi_{a}}}}j_{\psi_{a}}(\varphi)(g_{S_{\psi_{a}}}k^{S_{\psi_{a}}})\mathrm{d}k^{S_{\psi_{a}}},
\end{equation}
where $ \omega_{S_{\psi_{a}}} $ is the restriction of $ \omega $ to $ G(\mathbb{A}_{S_{\psi_{a}}}) $ and $ g_{S_{\psi_{a}}}\in G(\mathbb{A}_{S_{\psi_{a}}}) $. Together with (\ref{6.5}) and (\ref{6.6}) we have
\begin{equation}\label{omegacom}
\omega_{S_{\psi_{a}}}(g_{S_{\psi_{a}}})=\int_{\mathsf{K}^{S_{\psi_{a}}}}\Theta\left(\pi_{\psi_{a}}(g_{S_{\psi_{a}}}k^{S_{\psi_{a}}})\varphi\right)\mathrm{d}k^{S_{\psi_{a}}}=\int_{\mathsf{K}^{S_{\psi_{a}}}}\sum_{\gamma\in \mathbb{Q}}\left(\pi_{\psi_{a}}(g_{S_{\psi_{a}}}k^{S_{\psi_{a}}})\varphi\right)(\gamma)\mathrm{d}k^{S_{\psi_{a}}}.
\end{equation}
Let
\begin{equation*}
\left( g(x_{v},z_{v},y_{v})\right)_{v\in S_{\psi_{a}}}\in  G(\mathbb{A}_{S_{\psi_{a}}}),
\end{equation*}
\begin{equation*}
\left( g(x_{k_{v}},z_{k_{v}},y_{k_{v}})\right)_{v\in S_{\psi_{a}}}\in  \mathsf{K}^{S_{\psi_{a}}},
\end{equation*}
then the multiplication is
\begin{equation*}
g(x_{v},z_{v},y_{v})\cdot g(x_{k_{v}},z_{k_{v}},y_{k_{v}})
=g(x_{v}+x_{k_{v}},z_{v}+z_{k_{v}}+x_{v}y_{k_{v}},y_{v}+y_{k_{v}}).
\end{equation*}
Recall from \cite[p.13, (3.6)]{ST04} that 
\begin{equation}\label{1}
\pi_{\psi_{a}}(g(0,0,y))\varphi(x)=\psi_{a}(y\cdot x)\varphi(x),
\end{equation}
\begin{equation}\label{2}
\pi_{\psi_{a}}(g(0,z,0))\varphi(x)=\psi_{a}(z)\varphi(x),
\end{equation}
\begin{equation}\label{3}
\pi_{\psi_{a}}(g(x',0,))\varphi(x)=\varphi(x+x').
\end{equation}
\cite[p.13, (3.6)]{ST04} is stated for $ \pi_{\psi_{a}}' $ which is unitarily equivalent to $ \pi_{\psi_{a}} $. This holds also for $ \pi_{\psi_{a}} $ since we can identify them in this setting.
\begin{lemm}\label{omega}
We have 
\begin{equation*}
\omega_{S_{\psi_{a}}}(g_{S_{\psi_{a}}})=\sum_{\gamma\in \mathbb{Z}_{S_{\psi_{a}}}}\prod_{v\in S_{\psi_{a}}}\psi_{a}(y_{v}\gamma+z_{v})\varphi(x_{v}+\gamma),
\end{equation*}
where $ \mathbb{Z}_{S_{\psi_{a}}} $ is the set of $ S_{\psi_{a}} $-integers.
\end{lemm}
\begin{proof}
By (\ref{omegacom}) and (\ref{1})-(\ref{3}) we have
\begin{equation*}
\omega_{S_{\psi_{a}}}(g_{S_{\psi_{a}}})
\end{equation*}
\begin{eqnarray*}
&=&\sum_{\gamma\in \mathbb{Q}}\prod_{v\in\Val(\mathbb{Q})}\int_{\mathsf{K}^{S_{\psi_{a}}}}\psi_{a}\left((y_{v}+y_{k_{v}})\gamma+z_{v}+z_{k_{v}}+x_{v}y_{k_{v}}\right)\varphi(x_{v}+x_{k_{v}}+\gamma)\mathrm{d}x_{k_{v}}\mathrm{d}y_{k_{v}}\mathrm{d}z_{k_{v}}\\
&=&\sum_{\gamma\in \mathbb{Q}}\left(\prod_{v\in S_{\psi_{a}}}\psi_{a}(y_{v}\gamma+z_{v})\varphi(x_{v}+\gamma)\cdot \prod_{v\not\in S_{\psi_{a}}}\int_{\mathsf{K}^{S_{\psi_{a}}}}\psi_{a}(y_{k_{v}}\gamma+z_{k_{v}})\varphi(x_{k_{v}}+\gamma)\mathrm{d}k^{S_{\psi_{a}}}\right).
\end{eqnarray*}
When $ \gamma\not\in \mathbb{Z}_{S_{\psi_{a}}}$,
\begin{equation*}
\int_{\mathsf{K}^{S_{\psi_{a}}}}\psi_{a}(y_{k_{v}}\gamma+z_{k_{v}})\varphi(x_{k_{v}}+\gamma)\mathrm{d}k^{S_{\psi_{a}}}=0,
\end{equation*}
therefore
\begin{equation*}
\omega_{S_{\psi_{a}}}(g_{S_{\psi_{a}}})=\sum_{\gamma\in \mathbb{Z}_{S_{\psi_{a}}}}\prod_{v\in S_{\psi_{a}}}\psi_{a}(y_{v}\gamma+z_{v})\varphi(x_{v}+\gamma).
\end{equation*}
\end{proof}

\begin{prop}\label{propM}
\begin{enumerate}
\item The function
\begin{equation*}
\mathbf{s}\mapsto \mathsf{Z}_{2,\varepsilon,S_{\psi_{a}}}(\mathbf{s},\omega^{\psi_{a}}):=\int_{G(\mathbb{A}_{S_{\psi_{a}}})}\mathsf{H}(\mathbf{s},g_{S_{\psi_{a}}})^{-1}\overline{\omega}_{S_{\psi_{a}}}^{\psi_{a}}(g_{S_{\psi_{a}}})\delta_{\varepsilon}(g_{S_{\psi_{a}}})\mathrm{d}g_{S_{\psi_{a}}}
\end{equation*}
is holomorphic in the domain $ \Re(s_{\alpha})>\kappa_{\alpha}-1+\delta $ for all $ \alpha\in \mathcal{A} $ and $ \delta>0 $. Moreover for any compact subset $ \mathcal{K} $ in the above domain, there exists $ n'>0 $ such that for any $ n>0 $ and $ \mathbf{s}\in\mathcal{K} $,
\begin{equation}\label{eestimate2}
\left| \mathsf{Z}_{2,\varepsilon,S_{\psi_{a}}}(\mathbf{s},\omega^{\psi_{a}})\right| \ll_{n,\mathcal{K}} \dfrac{(1+\vert a\vert)^{n'}}{(1+\vert\lambda\vert)^{n}}
\end{equation}
where $ \lambda=\lambda(\omega_{S_{\psi_{a}}}^{\psi_{a}}) $ is the eigenvalue of $ \omega_{S_{\psi_{a}}}^{\psi_{a}} $ described in the proof of \cite[Proposition 4.12]{ST04}. 
\item If $ L=\sum_{\alpha\in \mathcal{A}}\lambda_{\alpha}D_{\alpha} $ is a big divisor on $ X $, let
\begin{equation*}
a':=\tilde{a}'((X,D_{\mathrm{red}}),L),\quad b':=b(\mathbb{Q}_{v},(X,D_{\mathrm{red}}),L,f_{a})
\end{equation*}
be as defined in \cite[\S 4]{PSTVA19}. Then there is a constant $ \delta>0 $ such that the function
\begin{equation*}
s\mapsto \prod_{v\in S_{\psi_{a}}}\zeta_{\mathbb{Q}_{v}}(s-a')^{-b'}\mathsf{Z}_{2,\varepsilon,S_{\psi_{a}}}(sL,\omega^{\psi_{a}})
\end{equation*}
admits a holomorphic continuation to the domain  $ \Re(s)>a'-\delta $. Moreover for any compact subset $ \mathcal{K} $ in the above domain, there exists $ n'>0 $ such that for any $ n>0 $ and $ \mathbf{s}\in\mathcal{K} $,
\begin{equation}\label{estimate2}
\left| \left(\prod_{v\in S_{\psi_{a}}}\zeta_{\mathbb{Q}_{v}}(s-a')^{-b'}\right)\mathsf{Z}_{2,\varepsilon,S_{\psi_{a}}}(sL,\omega^{\psi_{a}})\right| \ll_{n,\mathcal{K}}\dfrac{(1+\vert a\vert)^{n'}}{(1+\vert\lambda\vert)^{n}}.
\end{equation}
\end{enumerate}
\end{prop}

\begin{proof}
\begin{enumerate}
\item The assertion on holomorphy follows from the proof of (2). We adopt the estimate (\ref{eestimate2}) from \cite[Proposition 4.12]{ST04}. Note that although \cite[Proposition 4.12]{ST04} is stated for $ \Re(s_{\alpha})>\kappa_{\alpha}-1/2+\delta $ where $ \delta>0 $, it holds also for $ \Re(s_{\alpha})>\kappa_{\alpha}-1 $ since the proof of \cite[Proposition 4.12]{ST04} relies on \cite[Lemma 4.1, Lemma 4.11]{ST04}. On the other hand, the proof of \cite[Lemma 4.11]{ST04} relies on \cite[Lemma 4.1, Proposition 4.4]{ST04} which hold for $ \Re(s_{\alpha})>\kappa_{\alpha}-1 $. 
\item By Lemma \ref{omega} we have
\begin{equation*}
\mathsf{Z}_{2,\varepsilon,S_{\psi_{a}}}(sL,\omega^{\psi_{a}})=\sum_{\gamma\in \mathbb{Z}_{S_{\psi_{a}}}}\int_{G(\mathbb{A}_{S_{\psi_{a}}})}\mathsf{H}(sL,g_{S_{\psi_{a}}})^{-1}\overline{\psi_{a}}(y\gamma+z)\varphi(x+\gamma)\delta_{\varepsilon}(g_{S_{\psi_{a}}})\mathrm{d}g_{S_{\psi_{a}}}
\end{equation*}
where $ \varphi $ is a Schwartz-Bruhat function as in \cite[\S 3.7]{ST04}. By Lemma \ref{k-invariant}, $ \mathsf{H}(sL,g) $ is $ \mathsf{K} $-invariant, in particular the $ v $-component $ \mathsf{H}_{v}(sL,g_{v}) $ is $ \mathsf{K}_{v} $-invariant for $ v\in S_{\psi_{a}} $. Therefore there is a $ \mathbb{Z} $-lattice $ \Lambda_{3}\subset \mathbb{Z}_{S_{\psi_{a}}} $ such that for all $ \gamma\not\in \Lambda_{3} $,
\begin{equation*}
\int_{G(\mathbb{A}_{S_{\psi_{a}}})}\mathsf{H}(sL,g_{S_{\psi_{a}}})^{-1}\overline{\psi_{a}}(y\gamma+z)\varphi(x+\gamma)\delta_{\varepsilon}(g_{S_{\psi_{a}}})\mathrm{d}g_{S_{\psi_{a}}}=0.
\end{equation*}
Thus
\begin{equation*}
\mathsf{Z}_{2,\varepsilon,S_{\psi_{a}}}(sL,\omega^{\psi_{a}})=\sum_{\gamma\in \Lambda_{3}}\int_{G(\mathbb{A}_{S_{\psi_{a}}})}\mathsf{H}(sL,g_{S_{\psi_{a}}})^{-1}\overline{\psi_{a}}(y\gamma+z)\varphi(x+\gamma)\delta_{\varepsilon}(g_{S_{\psi_{a}}})\mathrm{d}g_{S_{\psi_{a}}}.
\end{equation*}
Since the Schwartz-Bruhat function $ \varphi $ is a finite sum of pure tensors, we may assume that $ \varphi $ is itself a pure tensor, i.e., $ \varphi=\prod_{v\in S_{\psi_{a}}}\varphi_{v} $. Also we may assume $ \varphi_{\infty} $ is an eigenfunction for $ \vartriangle_{\psi_{a}} $, the Laplace operator associated with $ \psi_{a} $ (see \cite[\S 3.5]{ST04}). Now we have
\begin{equation*}
\int_{G(\mathbb{A}_{S_{\psi_{a}}})}\mathsf{H}(sL,g_{S_{\psi_{a}}})^{-1}\overline{\psi_{a}}(y\gamma+z)\varphi(x+\gamma)\delta_{\varepsilon}(g_{S_{\psi_{a}}})\mathrm{d}g_{S_{\psi_{a}}}
\end{equation*}
\begin{equation*}
=\prod_{v\in S_{\psi_{a}}\backslash\{\infty\}}\int_{G(\mathbb{Q}_{v})}\mathsf{H}_{v}(sL,g_{v})^{-1}\overline{\psi_{a,v}}(y_{v}\gamma+z_{v})\varphi_{v}(x_{v}+\gamma)\delta_{\varepsilon,v}(g_{v})\mathrm{d}g_{v}
\end{equation*}
\begin{equation}\label{archi}
\times\int_{G(\mathbb{R})}\mathsf{H}_{\infty}(sL,g_{\infty})^{-1}\overline{\psi_{a,\infty}}(y_{\infty}\gamma+z_{\infty})\varphi_{\infty}(x_{\infty}+\gamma)\mathrm{d}g_{\infty}.
\end{equation}
Integration by parts $ N $ times, the integral (\ref{archi}) is then majorized by
\begin{equation*}
\dfrac{1}{(1+\vert a\gamma\vert)^{N}}\dfrac{1}{\vert \lambda_{n}^{\psi_{a}}\vert^{N}}\left|\int_{G(\mathbb{R})}\left(\vartriangle_{\psi_{a}}^{N}\dfrac{\partial^{N}}{\partial y_{\infty}^{N}}\mathsf{H}(sL,g_{\infty})^{-1}\right)\varphi_{\infty}(x_{\infty}+\gamma)\overline{\psi_{a,\infty}}(y_{\infty}\gamma+z_{\infty})\mathrm{d}g_{\infty}\right|,
\end{equation*}
where $ \lambda_{n}^{\psi_{a}} $ is the eigenvalue of $ \vartriangle_{\psi_{a}} $ (see \cite[\S 3.5]{ST04}). Combining with \cite[Lemma 9.7]{ST15}, we conclude that for large $ N $,
\begin{equation*}
\mathsf{Z}_{2,\varepsilon,\infty}(sL,\omega^{\psi_{a}}):=\sum_{\gamma\in \Lambda_{3}}\int_{G(\mathbb{R})}\mathsf{H}_{\infty}(sL,g_{\infty})^{-1}\overline{\psi_{a,\infty}}(y_{\infty}\gamma+z_{\infty})\varphi_{\infty}(x_{\infty}+\gamma)\mathrm{d}g_{\infty}
\end{equation*}
is convergent.\par
Now the assertion on holomorphic continuation  follows from results on oscillatory integrals \cite[\S 3.4]{CT-additive} and the proof is analogous to that of Proposition \ref{propJ} (2). As Proposition \ref{propJ} (2) and \cite[Proposition 3.4.4]{CT-additive}, the holomorphic continuation 
is insured by the positivity $ d_{\alpha}(f_{a})>0 $ for some $ \alpha\in \mathcal{A}^{Y} $. We adopt the estimate (\ref{estimate2}) from \cite[Proposition 4.12, Lemma 4.14]{ST04}, as in the proof of (1), \cite[Proposition 4.12, Lemma 4.14]{ST04} actually hold also for $ \Re(s_{\alpha})>\kappa_{\alpha}-1 $. The estimate (\ref{estimate2}) still holds after holomorphic continuation.
\end{enumerate}
\end{proof}

\subsection{Euler products}
Finally we consider the Euler product from Lemma \ref{lemma decomposition}:
\begin{equation*}
\mathsf{Z}_{2,\varepsilon}(\mathbf{s},\omega^{\psi_{a}}):=\prod_{p\not \in S_{\psi_{a}}}\mathsf{Z}_{2,\varepsilon,p}(\mathbf{s},\omega^{\psi_{a}})\cdot \mathsf{Z}_{2,\varepsilon,S_{\psi_{a}}}(\mathbf{s},\omega^{\psi_{a}}).
\end{equation*}
For every $ \alpha\in \mathcal{A}^{Y} $ we set
\begin{equation*}
\zeta_{F_{\alpha}^{Y},S_{\psi_{a}}}(s)=\prod_{v\not \in S_{\psi_{a}}}\prod_{\beta\in \mathcal{A}_{v}^{Y}(\alpha)}\zeta_{\mathbb{Q}_{v,\beta}^{Y}}(s).
\end{equation*}

\begin{prop}\label{propQ}
Assume that $ (X,D_{\varepsilon}) $ is klt. Then there is a constant $ \delta>0 $, independent of $ a $, such that the following function
\begin{equation*}
\mathbf{s}\mapsto \left(\prod_{\alpha\in \mathcal{A}^{0}}\zeta_{F_{\alpha}^{Y},S_{\psi_{a}}}(m_{\alpha}(s_{\alpha}-\kappa_{\alpha}^{Y}+1))\right)^{-1}\mathsf{Z}_{2,\varepsilon}(\mathbf{s},\omega^{\psi_{a}})
\end{equation*}
is holomorphic on the domain $ \mathsf{T}_{>-\delta} $.\par 
Moreover, for any compact subset $ \mathcal{K} $ in $ \mathsf{T}_{>-\delta} $, there exists $ n'>0 $ such that for any $ n>0 $ and $ \mathbf{s}\in\mathcal{K} $,
\begin{equation*}\label{estimate4}
\left| \left(\prod_{\alpha\in \mathcal{A}^{0}}\zeta_{F_{\alpha}^{Y},S_{\psi_{a}}}(m_{\alpha}(s_{\alpha}-\kappa_{\alpha}^{Y}+1))\right)^{-1}\mathsf{Z}_{2,\varepsilon}(\mathbf{s},\omega^{\psi_{a}})\right| \ll_{n,\mathcal{K}} \dfrac{(1+\vert a\vert)^{n'}}{(1+\vert \lambda\vert)^{n}}
\end{equation*}
where $ \lambda=\lambda(\omega_{S_{\psi_{a}}}^{\psi_{a}}) $ is the eigenvalue of $ \omega_{S_{\psi_{a}}}^{\psi_{a}} $ described in the proof of \cite[Proposition 4.12]{ST04}..
\end{prop}
\begin{proof}
This follows from Propositions \ref{propK}, \ref{propL}, \ref{propLL} and \ref{propM}, together with the equality (\ref{eqE}).
\end{proof}

\section{Proof of main results}\label{section proof}
In this section we prove our main results. We denote by $ \overline{X} $ the base change of $ X $ to $ \overline{\mathbb{Q}} $ and we write $ X_{v} $ for the base change of $ X $ to $ \mathbb{Q}_{v} $. \par
Let $ \mathcal{L} $ be a big line bundle $ L $ on $ X $ equipped with a smooth adelic metrization. We are concerned with the asymptotic behavior of the counting function
\begin{equation*}
\mathsf{N}(G(\mathbb{Q})_{\varepsilon},\mathcal{L},T)
\end{equation*}
where
\begin{equation*}
G(\mathbb{Q})_{\varepsilon}=G(\mathbb{Q})\cap (\mathcal{X}, \cD_{\varepsilon})(\mathbb{Z}_S).
\end{equation*}

\subsection{Proof of Theorem \ref{main thm}}
In this subsection we assume that
the pair $ (X,D_{\varepsilon}) $ is klt.
\begin{prop}\label{propR}
The following function
\begin{equation*}
\mathbf{s}\mapsto \left(\prod_{\alpha\in \mathcal{A}}\zeta_{F_{\alpha}}(m_{\alpha}(s_{\alpha}-\kappa_{\alpha}+1))\right)^{-1}\mathsf{Z}_{\varepsilon}(\mathbf{s})
\end{equation*}
is holomorphic on $ \mathsf{T}_{> 0} $. 
\end{prop}
\begin{proof}
It follows from Propositions \ref{propN} and \ref{propO} that $ \mathsf{Z}_{0,\varepsilon}(\mathbf{s}) $ is holomorphic on $ \mathsf{T}_{> 0} $. We now show the convergence of
\begin{equation*}
\mathsf{Z}_{1,\varepsilon}(\mathbf{s})=\sum_{\mathbf{a}\in\Lambda_{1}}\mathsf{Z}_{1,\varepsilon}(\mathbf{s},\eta_{\mathbf{a}}).
\end{equation*}
It follows from Proposition \ref{propP} that $  \mathsf{Z}_{1,\varepsilon}(\mathbf{s},\eta_{\mathbf{a}}) $ is holomorphic for $ \mathbf{s}\in \mathsf{T}_{>0} $ and
\begin{equation*}
\vert \mathsf{Z}_{1,\varepsilon}(\mathbf{s},\eta_{\mathbf{a}})\vert \ll_{n,\mathcal{K}}\dfrac{1}{(1+H_{\infty}(\mathbf{a}))^{n}},
\end{equation*}
where $ \mathcal{K} $ is a compact subset of $ \mathsf{T}_{>0} $ containing $ \mathbf{s} $. As the series
\begin{equation*}
\sum_{\mathbf{a}\in \Lambda_{1}} \dfrac{1}{(1+H_{\infty}(\mathbf{a}))^{n}}
\end{equation*}
converges for sufficiently large $ n $, the convergence of $ \mathsf{Z}_{1,\varepsilon}(\mathbf{s}) $ follows.\par
Next we show the convergence of
\begin{equation*}
\mathsf{Z}_{2,\varepsilon}(\mathbf{s})=\sum_{a\in \Lambda_{2}}\sum_{\omega^{\psi_{a}}\in\mathcal{B}(\varrho_{\psi_{a}})}\omega^{\psi_{a}}(e)\mathsf{Z}_{2,\varepsilon}(\mathbf{s},\omega^{\psi_{a}}).
\end{equation*}
It follows from Proposition \ref{propQ} and the estimate (\ref{4}) that $ \mathsf{Z}_{2,\varepsilon}(\mathbf{s},\omega^{\psi_{a}}) $ is holomorphic for $ \mathbf{s}\in \mathsf{T}_{>0} $ and
\begin{equation*}
\vert\mathsf{Z}_{2,\varepsilon}(\mathbf{s},\omega^{\psi_{a}})\vert \ll_{n,\mathcal{K}}\dfrac{(1+\vert a\vert)^{n'}}{(1+\vert \lambda\vert)^{n}},
\end{equation*}
where $ a\in \Lambda_{2} $, $ \mathcal{K} $ is a compact subset of $ \mathsf{T}_{>0} $ containing $ \mathbf{s} $, $ n $ and $ n' $ are as in Proposition \ref{propQ}. As in the proof of \cite[Theorem 4.15]{ST04}, it then suffices to prove the convergence of the series
\begin{equation}\label{series}
\sum_{a\in \Lambda_{2}}\sum_{\omega^{\psi_{a}}}\vert \lambda'\vert^{-n}\vert a\vert^{n'}
\end{equation}
where $ \lambda'=\lambda'(\omega^{\psi_{a}}) $ is the eigenvalue of $ \omega^{\psi_{a}} $. The rest of the proof of the convergence of $ \mathsf{Z}_{2,\varepsilon}(\mathbf{s}) $ is the same as that in \cite[Theorem 4.15]{ST04}. Namely, the series (\ref{series}) is then bounded from above by
\begin{equation*}
\sum_{a\in \mathbb{Z},a\neq 0}\sum_{n}\vert \lambda_{n}'\vert^{-n}\vert a\vert^{n'+1}\cdot n(\mathsf{K})^{2}
\end{equation*}
where $ \lambda_{n}'=(-2\pi(n+1)\vert a\vert -4\pi^{2}a^{2}) $ and see (\ref{nk}) for $ n(\mathsf{K}) $. Therefore the series (\ref{series}) is convergent.\par
We conclude that the spectral decomposition 
\begin{equation*}
\mathsf{Z}_{\varepsilon}(\mathbf{s})=\mathsf{Z}_{0,\varepsilon}(\mathbf{s})+\mathsf{Z}_{1,\varepsilon}(\mathbf{s})+\mathsf{Z}_{2,\varepsilon}(\mathbf{s})
\end{equation*}
holds for $ \Re(\mathbf{s})\gg 0 $. Now the proposition follows from Proposition \ref{propN}, Corollary \ref{propO}, Proposition \ref{propP} and Proposition \ref{propQ}.
\end{proof}

We now discuss the case where $ \mathbf{s}=sL $. Write $ L=\sum_{\alpha\in \mathcal{A}}\lambda_{\alpha}D_{\alpha} $ where $ \lambda_{\alpha}>0 $ for all $ \alpha\in \mathcal{A} $, then $ s_{\alpha}=s\lambda_{\alpha} $. It follows from Proposition \ref{propR} that the rightmost pole along $ \Re(s) $ of $ \mathsf{Z}_{\varepsilon}(sL) $ is
\begin{equation*}
\overline{a}=a((X,D_{\varepsilon}),L)=\max_{\alpha\in \mathcal{A}}\left\lbrace \dfrac{\kappa_{\alpha}-\varepsilon_{\alpha}}{\lambda_{\alpha}}\right\rbrace.
\end{equation*}
We set
\begin{equation*}
\mathcal{A}_{\varepsilon}(L)=\left\lbrace \alpha\in \mathcal{A}: \dfrac{\kappa_{\alpha}-\varepsilon_{\alpha}}{\lambda_{\alpha}}=\overline{a}\right\rbrace ,
\end{equation*}
and 
\begin{equation*}
\overline{b}=b(\mathbb{Q},(X,D_{\varepsilon}),L):=\#\mathcal{A}_{\varepsilon}(L).
\end{equation*}
Assume that the divisor $ \overline{a}L+K_{X}+D_{\varepsilon} $ is rigid, by which we mean that its Iitaka dimension is 0. Because of the spectral decomposition 
\begin{equation*}
\mathsf{Z}_{\varepsilon}(sL)=\mathsf{Z}_{0,\varepsilon}(sL)+\mathsf{Z}_{1,\varepsilon}(sL)+\mathsf{Z}_{2,\varepsilon}(sL),
\end{equation*}
we shall investigate the poles of $ \mathsf{Z}_{\varepsilon}(sL) $  by investigating individually $ \mathsf{Z}_{0,\varepsilon}(sL) $, $ \mathsf{Z}_{1,\varepsilon}(sL) $ and $ \mathsf{Z}_{2,\varepsilon}(sL) $.\par 
For the term $ \mathsf{Z}_{1,\varepsilon}(sL) $, Proposition \ref{propP} shows that $ \mathsf{Z}_{1,\varepsilon}(sL,\eta_{\mathbf{a}}) $ has a pole at $ \overline{a} $ of the highest order
equal to that of $ \mathsf{Z}_{0,\varepsilon}(sL) $ if and only if for $ \mathbf{a}\in \Lambda_{1} $,
\begin{equation*}
\mathcal{A}^{0}(\mathbf{a})\supset \mathcal{A}_{\varepsilon}(L),
\end{equation*}
which means that $ d_{\alpha}(f_{\mathbf{a}})=0 $ whenever $ (\kappa_{\alpha}-\varepsilon_{\alpha})/\lambda_{\alpha}=\overline{a} $. Since
\begin{equation*}
E(f_{\mathbf{a}})\sim \sum_{\alpha\in \mathcal{A}}d_{\alpha}(f_{\mathbf{a}})D_{\alpha},\quad  \overline{a}L+K_{X}+D_{\varepsilon}=\sum_{\alpha\in \mathcal{A}}(\overline{a}\lambda_{\alpha}-\kappa_{\alpha}+\varepsilon_{\alpha})D_{\alpha},
\end{equation*}
this means that $ E(f_{\mathbf{a}}) $ is equivalent to a boundary divisor whose support is contained in that of the divisor $ \overline{a}L+K_{X}+D_{\varepsilon} $. This is impossible because $ \overline{a}L+K_{X}+D_{\varepsilon} $ is rigid. Similarly, Proposition \ref{propQ} shows that the term $ \mathsf{Z}_{2,\varepsilon}(sL) $ does not contribute to the main term of $ \mathsf{Z}_{\varepsilon}(sL) $.\par
On the other hand, it follows from Corollary \ref{propO} that $ \mathsf{Z}_{0,\varepsilon}(sL) $ has a pole at $ s=\overline{a} $ of order $ \overline{b} $ if we can show that the corresponding residue $ \overline{c} $ is nonzero, i.e.,
\begin{equation*}
\overline{c}:=\lim_{s\rightarrow \overline{a}}(s-\overline{a})^{\overline{b}}\mathsf{Z}_{0,\varepsilon}(sL)\neq 0.
\end{equation*}
Recall that
\begin{equation*}
\mathsf{Z}_{0,\varepsilon}(sL)=\int_{G(\mathbb{A}_{\mathbb{Q}})}\mathsf{H}(sL,g)^{-1}\delta_{\varepsilon}(g)\mathrm{d}g=\int_{G(\mathbb{A}_{\mathbb{Q}})_{\varepsilon}}\mathsf{H}(sL+K_{X},g)^{-1}\mathrm{d}\tau,
\end{equation*}
where $ \tau $ is the Tamagawa measure on $ G $. Let 
\begin{equation*}
X^{\circ}=X\backslash \left(\cup_{\alpha\not\in \mathcal{A}_{\varepsilon}(L)}D_{\alpha}\right),
\end{equation*}
and let $ \tau_{X^{\circ}} $ denote the Tamagawa measure on $ X^{\circ} $. We also define
\begin{equation*}
\tau_{X^{\circ},D_{\varepsilon}}=\mathsf{H}(D_{\varepsilon},g)\tau_{X^{\circ}}.
\end{equation*}

\begin{lemm}\label{lemma}
Let the notations be as above. We have
\begin{equation*}
\overline{c}=\prod_{\alpha\in \mathcal{A}_{\varepsilon}(L)}\dfrac{1}{m_{\alpha}\lambda_{\alpha}}\int_{X^{\circ}(\mathbb{A}_{\mathbb{Q}})_{\varepsilon}}\mathsf{H}(aL+K_{X}+D_{\varepsilon},g)^{-1}\mathrm{d}\tau_{X^{\circ},D_{\varepsilon}}>0.
\end{equation*}
where $ X^{\circ}(\mathbb{A}_{\mathbb{Q}})_{\varepsilon}=\prod'_{v\in\Val(\mathbb{Q})}X^{\circ}(\mathbb{Q}_{v})_{\varepsilon} $.
\end{lemm}
\begin{proof}
The proof is essentially analogous to that of \cite[Lemma 9.3]{PSTVA19}.
\end{proof}

Applying a Tauberian theorem \cite[II.7, Theorem 15]{Ten} we have the following result, which agrees with Conjecture \ref{conj}.
\begin{theo}
Let $ \mathcal{X},\mathcal{L},\mathcal{D},\overline{a},\overline{b},\overline{c} $ and $ \varepsilon $ be as above. Assume that $ (X,D_{\varepsilon}) $ is klt. If $ \overline{a}L+K_{X}+D_{\varepsilon} $ is rigid, then as $ T\rightarrow \infty $,
\begin{equation*}
\mathsf{N}(G(\mathbb{Q})_{\varepsilon},\mathcal{L},T)\sim \dfrac{\overline{c}}{\overline{a}(\overline{b}-1)!}T^{\overline{a}}(\log T)^{\overline{b}-1}.
\end{equation*}
\end{theo}

\subsection{Proof of Theorem \ref{main thm2}}
Let notations be as in the former subsection but in this subsection we assume that
the pair $ (X,D_{\varepsilon}) $ is only dlt.
We set
\begin{equation*}
\mathcal{A}^{\mathrm{klt}}=\{\alpha\in \mathcal{A}\vert \varepsilon_{\alpha}\neq 1\},
\end{equation*}
\begin{equation*}
\mathcal{A}^{\mathrm{nklt}}=\{\alpha\in \mathcal{A}\vert \varepsilon_{\alpha}=1\}.
\end{equation*}

\begin{prop}\label{propS}
Let $ L=-\left(K_{X}+D_{\varepsilon}\right) $. Then the function
\begin{equation*}
s\mapsto \left(\prod_{\alpha\in \mathcal{A}^{\mathrm{klt}}}\zeta_{F_{\alpha}}(1+m_{\alpha}(\kappa_{\alpha}-\varepsilon_{\alpha})(s-1))\right)^{-1}\left(\prod_{v\in S}\zeta_{\mathbb{Q}_{v}}(s-1)^{-b(\mathbb{Q}_{v},(X,D_{\mathrm{red}}),L)}\right) \mathsf{Z}_{\varepsilon}(sL)
\end{equation*}
is holomorphic when $ \Re(s)>1-\delta $ for some $ \delta>0 $ where $ b(\mathbb{Q}_{v},(X,D_{\mathrm{red}}),L) $ is the $ b $-invariant defined in \cite[\S 4]{PSTVA19}.
\end{prop}
\begin{proof}
For $ \alpha\in \mathcal{A}^{\mathrm{klt}} $ we have $ sL=-s\left(K_{X}+D_{\varepsilon}\right):=(s_{\alpha})_{\alpha\in \mathcal{A}^{\mathrm{klt}}} $, thus we have $ s_{\alpha}=s(\kappa_{\alpha}-\varepsilon_{\alpha}) $. Recall that $ m_{\alpha}=\frac{1}{1-\varepsilon_{\alpha}} $, therefore in Proposition \ref{propR}, we have $ m_{\alpha}(s_{\alpha}-\kappa_{\alpha}+1)=1+m_{\alpha}(\kappa_{\alpha}-\varepsilon_{\alpha})(s-1) $ for $ \alpha\in \mathcal{A}^{\mathrm{klt}} $.\par 
For $ \alpha\in \mathcal{A}^{\mathrm{nklt}} $, this case corresponds to $ v\in S $. Since $ L=-\left(K_{X}+D_{\varepsilon}\right) $, in Conjecture\ref{conj} or Proposition \ref{propN}, $ a=1 $. Thus the corresponding factor is
\begin{equation}\label{factor}
\prod_{v\in S}\zeta_{\mathbb{Q}_{v}}(s-1)^{-b(\mathbb{Q}_{v},(X,D_{\mathrm{red}}),L)}.
\end{equation}
The rest of the proof is analogous to that of Proposition \ref{propR}, i.e., using same arguments in the proof of Proposition \ref{propR} we see that the spectral decomposition 
\begin{equation*}
\mathsf{Z}_{\varepsilon}(sL)=\mathsf{Z}_{0,\varepsilon}(sL)+\mathsf{Z}_{1,\varepsilon}(sL)+\mathsf{Z}_{2,\varepsilon}(sL)
\end{equation*}
holds when $ \Re(s)>1-\delta $ for some $ \delta>0 $. Now the proposition follows from Proposition \ref{propN}, Corollary \ref{propO}, Proposition \ref{propP} and Proposition \ref{propQ} together with the factor (\ref{factor}).
\end{proof}
The proposition above implies that $ \mathsf{Z}_{\varepsilon}(sL) $ has a pole at $ s=1 $.\par
We define
\begin{equation*}
b'(\mathbb{Q},S,(X,D_{\varepsilon}),L)=\#\mathcal{A}^{\mathrm{klt}}+\sum_{v\in S}b(\mathbb{Q}_{v},(X,D_{\mathrm{red}}),L). 
\end{equation*}
Let $ D_{\mathrm{red}}=\sum_{\alpha\in \mathcal{A}}D_{\alpha} $ and $ v $ a place of $ \mathbb{Q} $, fix an embedding $ \overline{\mathbb{Q}}\subset \overline{\mathbb{Q}}_{v} $ so that $ \Gamma_{v}:=\Gal(\overline{\mathbb{Q}}_{v}/\mathbb{Q}_{v}) $ acts on $ \overline{X}=X\otimes_{\mathbb{Q}}\overline{\mathbb{Q}} $ and $ \overline{D}_{\mathrm{red}}=D_{\mathrm{red}}\otimes_{\mathbb{Q}}\overline{\mathbb{Q}} $. We denote by $ \overline{\mathcal{A}} $ the indexing set of $ \overline{D}_{\mathrm{red}} $ and by $ \mathcal{A}_{v} $ the set of orbits of $ \overline{\mathcal{A}} $ under the action of $ \Gamma_{v} $.\par

By the definition of $ b(\mathbb{Q}_{v},(X,D_{\mathrm{red}}),L) $ in \cite[\S 4]{PSTVA19} we have
\begin{equation*}
b'(\mathbb{Q},S,(X,D_{\varepsilon}),L)=\#\mathcal{A}^{\mathrm{klt}}+\sum_{v\in S}\max_{B\subseteq \mathcal{A}_{v}^{\mathrm{nklt}}}\left\lbrace\#B; \bigcap_{\beta\in B}D_{v,\beta,\mathrm{red}}(\mathbb{Q}_{v})\neq \emptyset\right\rbrace .
\end{equation*}
\begin{lemm}\label{overlemma1}
The height zeta function $ \mathsf{Z}_{0,\varepsilon}(sL) $ has a pole at $ s=1 $ of order $ b' $. 
\end{lemm}
\begin{proof}
This follows from Proposition \ref{propN} and Corollary \ref{propO}.
\end{proof}

\begin{lemm}\label{overlemma2}
The order of the pole at $ s=1 $ of $ \mathsf{Z}_{1,\varepsilon}(sL) $ is strictly less than $ b' $.  
\end{lemm}
\begin{proof}
We denote the order of the pole at $ s=1 $ of $ \mathsf{Z}_{1,\varepsilon}(sL) $ by $ b_{1} $. It is clear that $ b_{1}\leq b'$. As the proof of \cite[Lemma 3.5.4]{CT-additive}, we prove our result by contradiction. Assume that $ b_{1}=b' $, then comparing the formulas of $ b_{1} $ and $ b' $ we have
\begin{enumerate}
\item $ d_{\alpha}(f_{\mathbf{a}})=0 $ for all $ \alpha\in\mathcal{A}^{\mathrm{klt}} $ and all $ \mathbf{a}\in\Lambda_{1} $,
\item for any $ v\in S $ there is a subset $ B\subseteq \mathcal{A}_{v}^{\mathrm{nklt}} $ of maximal cardinality with 
\begin{equation*}
\bigcap_{\beta\in B}D_{v,\beta,\mathrm{red}}(\mathbb{Q}_{v})\neq \emptyset
\end{equation*}
such that $ d_{\alpha}(f_{\mathbf{a}})=0 $ for all $ \alpha\in B $.
\end{enumerate}
Fix a $ g(b_{1},0,b_{2})\in G(\mathbb{Q}) $ such that $ f_{\mathbf{a}}\left( g(b_{1},0,b_{2})\right)=1 $. Let us fix a $ v\in S $ and let $ B\subseteq \mathcal{A}_{v}^{\mathrm{nklt}} $ satisfy condition (2). By definition the function $ f_{\mathbf{a}} $ is defined and nonzero at general points of $ \bigcap_{\beta\in B}D_{v,\beta,\mathrm{red}} $ and there is such a general point $ d\in\bigcap_{\beta\in B}D_{v,\beta,\mathrm{red}}(\mathbb{Q}_{v}) $ for a suitable $ B $. Taking $ d'=\lim\limits_{t\rightarrow \infty}g(tb_{1},0,tb_{2})\cdot d $ where $ \cdot $ means the group action, we have $ d'\in\bigcap_{\beta\in B}D_{v,\beta,\mathrm{red}}(\mathbb{Q}_{v}) $ since the set $ \bigcap_{\beta\in B}D_{v,\beta,\mathrm{red}}(\mathbb{Q}_{v}) $ is closed. The function $ t\mapsto f_{\mathbf{a}}(g(tb_{1},0,tb_{2})\cdot d) $ is defined on $ \mathbb{P}^{1}=\{(1:t)\} $ and $ f_{\mathbf{a}}(g(tb_{1},0,tb_{2})\cdot d)\rightarrow \infty $ as $ t\rightarrow \infty $. Therefore $ d'\in D_{\alpha} $ for some $ \alpha\in \mathcal{A} $ such that $ d_{\alpha}(f_{\mathbf{a}})>0 $. Then $ \alpha\in\mathcal{A}^{\mathrm{nklt}} $ by condition (1), $ \alpha\not\in B $ by condition (2), and thus $ B'=B \cup\{\alpha\} \subseteq \mathcal{A}_{v}^{\mathrm{nklt}} $ is such that $ \bigcap_{\beta\in B'}D_{v,\beta,\mathrm{red}}(\mathbb{Q}_{v})\neq \emptyset $, contradicting the maximal cardinality condition (2).
\end{proof}
\begin{lemm}\label{overlemma3}
The order of the pole at $ s=1 $ of $ \mathsf{Z}_{2,\varepsilon}(sL) $ is strictly less than $ b' $.  
\end{lemm}
\begin{proof}
The argument is analogous to that of Lemma \ref{overlemma2}.
\end{proof}

\begin{theo}
Let $ \mathcal{X},\mathcal{D},\varepsilon $ and $ b' $ be as above and let $ L=-\left(K_{X}+D_{\varepsilon}\right) $. Assume that $ (X,D_{\varepsilon}) $ is  dlt. Then as $ T\rightarrow \infty $, there is a constant $ \overline{c}>0 $ depending on $ S,(\mathcal{X},\mathcal{D}_{\varepsilon}) $ and $ \mathcal{L} $ such that
\begin{equation*}
\mathsf{N}(G(\mathbb{Q})_{\varepsilon},\mathcal{L},T)\sim \dfrac{\overline{c}}{(b'-1)!}T(\log T)^{b'-1}.
\end{equation*}
\end{theo}
\begin{proof}
This follows from Proposition \ref{propS}, Lemmas \ref{overlemma1}, \ref{overlemma2}, \ref{overlemma3} and a Tauberian theorem \cite[II.7, Theorem 15]{Ten}.
\end{proof}

\bibliographystyle{alpha}
\bibliography{References}

\end{document}